\documentclass[11pt]{amsart}

\usepackage{amsthm}
\usepackage{amssymb}
\usepackage{xcolor}

\usepackage{hyperref}
\usepackage{enumerate}
\usepackage[margin=1.0in]{geometry}
\usepackage{tikz}

\DeclareMathOperator{\mesh}{mesh}
\DeclareMathOperator{\diam}{diam}
\DeclareMathOperator{\Comp}{Comp}

\newcommand{\R}{\mathbb{R}}
\newcommand{\Z}{\mathbb{Z}}
\newcommand{\D}{\mathbb{D}}
\newcommand{\N}{\mathbb{N}}

\newtheorem{theorem}{Theorem}[section]

\newtheorem{definition}[theorem]{Definition}

\newtheorem{lemma}[theorem]{Lemma}

\newtheorem{proposition}[theorem]{Proposition}

\newtheorem{corollary}[theorem]{Corollary}

\newtheorem{remark}{Remark}

\title{Thermodynamic formalism for coarse expanding dynamical systems}

\author[T. Das]{Tushar Das}
\address{Department of Mathematics and Statistics,  University of Wisconsin - La Crosse, 1725 State Street,
La Crosse, WI 54601, USA}
\email{tdas@uwlax.edu}

\author[F. Przytycki]{Feliks Przytycki}
\address{Institute of Mathematics, Polish Academy of Sciences, ul. \'Sniadeckich 8, 00-656 Warszawa, Poland}
\email{feliksp@impan.pl}

\author[G. Tiozzo]{Giulio Tiozzo}
\address{Department of Mathematics, University of Toronto,  40 St George St, Toronto, ON M5S 2E4, Canada}
\email{tiozzo@math.utoronto.ca}

\author[M. Urba\'nski]{Mariusz Urba\'nski}
\address{Department of Mathematics, University of North Texas,  1155
  Union Circle \#311430, Denton, TX 76203-5017, USA}
  \email{urbanski@unt.edu}

\author[A. Zdunik]{Anna Zdunik}
\address{Institute of Mathematics, University of Warsaw, Banacha 2, 02-097 Warszawa, Poland}
\email{a.zdunik@mimuw.edu.pl}

\date{\today}

\begin{document}

\begin{abstract}
We consider a class of dynamical systems, which we call \emph{weakly coarse expanding}, which is a generalization 
to the postcritically infinite case of \emph{expanding Thurston} maps as discussed by Bonk--Meyer and
is closely related to \emph{coarse expanding conformal} systems as defined by Ha\"issinsky--Pilgrim.
We prove existence and uniqueness of equilibrium states for a wide class of potentials,
as well as statistical laws such as a central limit theorem, law of iterated logarithm,
exponential decay of correlations and a large deviation principle.
Further, if the system is defined on the $2$-sphere, we prove all such results even in presence of periodic (repelling) branch points.
\end{abstract}

\maketitle

\section{Introduction}

The study of dynamical systems through thermodynamic formalism goes back to Sinai \cite{Sinai}, Bowen \cite{bowen}, and Ruelle \cite{Ruelle}.
The paradigmatic example in this theory is the one-sided shift map on the space of sequences on a finite alphabet. This space is endowed with a natural
metric under which the shift is uniformly expanding. This feature is essential in most applications: indeed, various systems which possess a certain degree of expansion can be encoded by a shift of finite type (see \cite{PU}), and this technique has allowed to establish statistical laws for a large class of dynamical systems.

For instance, several authors (see \cite{Przytycki-ICM2018} and references therein) have addressed the case of rational maps, i.e. holomorphic endomorphisms of the $2$-sphere, and established statistical laws for them.

A related class of examples of continuous self-maps of the $2$-sphere is represented by \emph{expanding Thurston maps}.
Such maps are postcritically finite, and can be described by a finite set of ``cut and fold" rules \cite{CFP}; note that they need not be
Thurston equivalent to holomorphic maps. These maps are discussed by Bonk--Meyer \cite{BM}, who among other things construct the measure of maximal
entropy. In a series of papers (\cite{Li1}, \cite{Li2}, \cite{Li3}, \cite{li-book}), Li works out the thermodynamic formalism for expanding Thurston maps, with respect to the visual metric.

In a similar vein, Ha\"issinsky--Pilgrim \cite{HP} define more generally the concept of \emph{coarse expanding conformal} (cxc) system, developing
in an axiomatic way the notion of expansion for maps of general metric spaces. In particular, they prove for cxc systems existence and uniqueness of the measure of maximal entropy.

The goal of this paper is to develop the thermodynamic formalism (in particular, prove existence and uniqueness of equilibrium states and statistical laws) for a general class of dynamical systems, which we call \emph{weakly coarse expanding}.
These systems are continuous finite branched coverings of locally connected topological spaces which are expanding in a weak metric sense, and generalize most cases discussed by \cite{BM} and \cite{HP}. In particular, they need not be postcritically finite, and the metric we consider need not be the visual metric, but it could be more generally an exponentially contracting metric (see Section \ref{S:def} for the definitions). Let us stress that neither conformality, nor holomorphy or smoothness is assumed in this paper. In particular, periodic branch points may be repelling despite the local degree at them being bigger than $1$.

Our main results are contained in the following two theorems.

\begin{theorem} \label{T:main}
Let $f : W_1 \to W_0$ be a weakly coarse expanding dynamical system without periodic critical points, let $X$ be its repellor, $\rho$ an exponentially contracting metric on $X$ compatible with the topology, and let
$\varphi : (X, \rho) \to \mathbb{R}$ be a H\"older continuous function. Then:
\begin{enumerate}
\item there exists a unique equilibrium state $\mu_\varphi$ for $\varphi$ on $X$.
\item[] Let $\psi : (X, \rho) \to \mathbb{R}$ be a H\"older continuous observable, and denote
$$S_n \psi(x) := \sum_{k = 0}^{n-1} \psi(f^k(x)).$$
Then there exists the finite limit
$$\sigma^2 := \lim_{n \to \infty} \frac{1}{n} \int_X \left( S_n \psi(x) - n \int \psi \ d \mu_\varphi \right)^2 d \mu_\varphi \geq 0$$
such that the following statistical laws hold:
\item (Central Limit Theorem, CLT)
If $\sigma > 0$, we have for any $a < b$
$$
\mu_\varphi\left(\left\{ x \in X \ : \ \frac{S_n \psi(x)  - n \int_X \psi \ d \mu_\varphi}{\sqrt{n}} \in [a, b]\right\} \right) \longrightarrow \frac{1}{\sqrt{2 \pi \sigma^2}} \int_a^b e^{-t^2/2\sigma^2 } \ dt$$
as $n \to \infty$.
If $\sigma = 0$, one has convergence in probability to the Dirac $\delta$-mass at $0$.

\item (Law of Iterated Logarithm, LIL)
For $\mu_\varphi$-a.e. $x \in X$,
$$\limsup_{n \to \infty} \frac{S_n \psi(x) - n \int_X \psi \ d \mu_\varphi}{\sqrt{n \log \log n}} = \sqrt{2 \sigma^2}.$$

\item (Exponential Decay of Correlations, EDC)
There exist constants $\alpha > 0$ and $C \geq 0$ such that for any $\mu_\varphi$--integrable function $\chi : X \to \mathbb{R}$,
for any $\beta$-H\"older function $\psi :  X \to \mathbb{R}$, and for any $n \geq 0$,
$$\left| \int_X \psi \cdot (\chi \circ f^n) \ d\mu_\varphi  - \int_X \psi \ d\mu_\varphi \cdot \int_X \chi \ d\mu_\varphi \right| \leq C e^{-n \alpha} \Vert \underline{\chi}  \Vert_1 \cdot \Vert \underline{\psi} \Vert_\beta,$$
where $\underline{\chi} := \chi - \int_X \chi \ d\mu_\varphi$, and $\Vert \cdot \Vert_\beta$ is the $\beta$-H\"older norm.

\item (Large Deviations, LD) For every $t \in\R$, we have that
$$
\begin{aligned}
\lim_{n\to\infty}\frac1n \log\mu_\varphi\bigg(\Big\{x\in X: {\rm sgn}(t)S_n \psi(x)&\ge {\rm sgn}(t) n \int_X\psi\, d\mu_{\varphi+t\psi}\Big\}\bigg) \\
&= -t\int_X\psi\, d\mu_{\varphi+t\psi}+P_{top}(\varphi+t\psi)-P_{top}(\varphi).
\end{aligned}
$$
\item \label{itm:sigma0}
Moreover,
$\sigma = 0$ if and only if there exists a
continuous $u : X \to \mathbb{R}$ such that
$$\psi - \int_X \psi \ d \mu_\varphi = u \circ f - u.$$

\item \label{itm:cohom}
Finally, $\mu_{\varphi_1} = \mu_{\varphi_2}$ if and only if there exist $K \in \mathbb{R}$ and a continuous $u : X \to \mathbb{R}$
such that
$$\varphi_1 - \varphi_2 = u \circ f - u + K.$$
In \eqref{itm:sigma0} and \eqref{itm:cohom}, the function $u$ is H\"older continuous with respect to a visual metric.
\end{enumerate}

\end{theorem}

Just as in classical holomorphic dynamics, a special role is played by \emph{critical points}, i.e. points where the map is not locally injective, also
called \emph{branch points}.
A major source of difficulty in the study of weakly coarse expanding systems is the presence of periodic critical points in the repellor
(indeed, those systems do not satisfy the [Degree] condition, hence they are not coarse expanding conformal in the sense of \cite{HP}).

Our second result addresses this issue in case the underlying space is an open subset of the $2$-sphere
(which is the case considered by \cite{BM} and \cite{li-book}); there, our results also hold in the presence of periodic critical points as follows.

\begin{theorem} \label{T:main-sphere}
If $f : W_1 \to W_0$ is a weakly coarse expanding system and $W_0 \subseteq S^2$ 
 is an open subset of the $2$-sphere, with the Euclidean topology, then all claims (1)-(2)-(3)-(4)-(5)-(6)-(7) of Theorem \ref{T:main} hold even if there are periodic critical points.
\end{theorem}

Examples of weakly coarse expanding systems are given, other than in \cite{BM} and \cite{HP}, in \cite{HP-ex};
in particular, iterated function systems \cite{ERS}, skew products induced by some finitely generated semigroups of rational functions, in particular those given by Examples 17.1 and 17.2 in \cite{ASU}, maps on Sierpi\'nski carpets and gaskets, as well as expanding \emph{polymodials}
(\cite{BCM1}, \cite{BCM2}).

For holomorphic examples, note there is an abundance of non-hyperbolic, non-postcritically finite rational maps $f:\widehat{\mathbb{C}} \to \widehat{\mathbb{C}}$ which have recurrent critical points in the Julia set $X=J(f)$ and satisfy backward exponential contraction (as in Theorem \ref{T:exp-contr}) for the Riemann metric $\rho$. This contraction condition is equivalent to the so-called
\emph{Topological Collet-Eckmann (TCE)} condition  and is satisfied e.g. for a Lebesgue-positive measure set of non-hyperbolic recurrent real parameters $c$ for the family $f_c(z)=z^2+c$ (see \cite{CE}, \cite{Tsujii} and \cite{Przytycki-ICM2018}).
These maps are weakly coarse expanding but not topologically cxc, since the [Degree] condition fails because of recurrence of the critical point $z = 0$.

For rational maps $f:\widehat{\mathbb{C}} \to \widehat{\mathbb{C}}$ satisfying TCE, all the assertions (1)-(7) of Theorem \ref{T:main}
hold for all H\"older potentials $\varphi: J(f) \to \mathbb{R}$, see \cite[Corollary 1.2]{InoRiv} and \cite{ComRiv}.
In fact the claims (1)-(7) hold for all (not only TCE) rational maps $f$ provided $\varphi$ satisfies $P(f^n, S_n(\varphi))>n \sup \varphi$ for some $n\in\mathbb{N}$. See \cite[Section 3]{Przytycki-ICM2018} for references.

\smallskip
Note that all results in Theorems~\ref{T:main} and \ref{T:main-sphere} work for $\varphi$ and $\psi$ lying in the slightly more general class of \emph{topologically H\"older} functions, as defined in Section \ref{S:top-holder}.
Let us also remark that our proofs follow a different, and in a way simpler, route than \cite{li-book}: indeed, there the author directly applies the analytic methods
of \cite{PU} on the metric space $X$. In the present paper, on the other hand, we first encode the dynamics of the system by a geometric coding tree,
in the footsteps of \cite{Feliks}, obtaining a semiconjugacy with the shift space; then, our statistical laws follow easily from the corresponding ones
for the shift space.
As for the second part (proof of Theorem \ref{T:main-sphere}), our technique consists in blowing up periodic critical points (and their grand orbit)
to circles. The new space obtained is a Sierpi\'nski carpet, and the dynamics becomes truly coarse expanding (without periodic critical points) on it,
hence we can apply the same techniques of Theorem \ref{T:main} to this new space, obtaining the desired results.

We have been informed by Peter Ha\"issinsky that he independently obtained a statement similar to Theorem \ref{T:main}.

\smallskip
The paper is organized as follows. The first part deals with the case without periodic critical points.
In Section 2, we introduce the definition of weakly coarse expanding systems and construct the geometric coding trees.
In Section 3, we prove that entropy does not drop under the projection from the symbolic space and in Section 4
we prove existence and uniqueness of equilibrium states as well as the statistical laws, proving Theorem \ref{T:main} (1)-(5).

In the second part we deal with periodic critical points; in particular, in Section 5 we define a new space by blowing up the sphere
along the grand orbit of every periodic critical (branch) point and produce, using Frink's lemma, an exponentially contracting metric, which we use to prove Theorem \ref{T:main-sphere}
(1)-(5). In Section 6 we deal with the cohomological equation, proving (6) and (7) of Theorem \ref{T:main} and \ref{T:main-sphere}.

Finally, in Appendix A we provide an alternative proof by constructing explicitly an exponentially contracting metric on
the blown up space, and in Appendix B we prove that the blown up space embeds into the sphere.

\subsection*{Acknowledgements}

We thank Peter Ha\"issinsky, Zhiqiang Li, Daniel Meyer, and Kevin Pilgrim for useful conversations, as well as the anonymous referees for their feedback. G. T. was partially supported by NSERC and the Alfred P. Sloan Foundation.
A. Z. was supported by the National Science Centre, Poland, grant no. 2018/31/B/ST1/02495.

\section{Definitions} \label{S:def}

Let us start by introducing the fundamental definitions of the objects we are dealing with.
We follow \cite{HP} in several places, though with some notable changes.

\subsection{Finite branched coverings}

Let $f : Y \to Z$ be a continuous map between locally compact Hausdorff topological spaces.
Note that such spaces satisfy the 
$T_3$ separation axiom (see e.g. \cite{Engelking}, Chapter 3).

The \emph{degree} of $f$ is defined as
$$\textup{deg}(f) := \sup \{ \# f^{-1}(z) \ : \ z \in Z \}.$$
Given a point $y \in Y$, the \emph{local degree} of $f$ at $y$ is
$$\textup{deg}(f; y) := \inf_U \sup \#\{ f^{-1}(z) \cap U \ : \ z \in f(U)\}$$
where $U$ ranges over all open neighborhoods of $y$. A point $y$ is \emph{critical} if $\textup{deg}(f; y)  > 1$.

\begin{definition}
The map $f : Y \to Z$ is a \emph{finite branched cover} of degree $d$ if $\textup{deg}(f) = d < \infty$ and the
two following conditions hold:
\begin{enumerate}
\item
for any $z \in Z$,
$$\sum_{y \in f^{-1}(z)} \textup{deg}(f; y) = \textup{deg}(f)$$
\item
for any $y_0 \in Y$ there are compact neighborhoods $U, V$ of $y_0$ and $f(y_0)$ such that
$$\sum_{y \in U, f(y) = z} \textup{deg}(f; y) = \textup{deg}(f; y_0)$$
for all $z \in V$.
\end{enumerate}
\end{definition}

We define the \emph{branch set} as $B_f := \{ y \in Y \ : \ \textup{deg}(f; y) > 1\}$, and the set of \emph{branch values} as $V_f = f(B_f)$.
Note that $Z \setminus V_f$ is the set of \emph{principal values}, i.e. the values $z \in Z$ so that $\#f^{-1}(z) = d$.

\begin{lemma}[\cite{HP}, Lemma 2.1.2]
A finite branched cover is open, closed, onto and proper. Furthermore, $B_f$ and $V_f$ are closed and nowhere dense.
\end{lemma}

A simple corollary of the lemma is that $f$ is a local homeomorphism away from critical points; i.e.,
for any $y \notin B_f$ there exists an open set $U$ which contains $y$ and such that $f : U \to f(U)$ is a homeomorphism.

\begin{definition}
A topological space $\mathcal{X}$ is \emph{strongly path connected} if for any countable subset $S \subset \mathcal{X}$, the space $\mathcal{X} \setminus S$ is path connected.
\end{definition}

Note that we need such an assumption: for instance, a tree is path connected but it becomes disconnected if you remove any point.

\subsection{Coarse expanding dynamical systems}

Let $W_1, W_0$ be two locally compact, locally connected, strongly path connected, Hausdorff topological spaces, 
and suppose that $W_1$ is an open subset of $W_0$ and the closure of $W_1$ is compact.

We define a \emph{system} as a triple $(f,W_0,W_1)$, where $W_0$, $W_1$ satisfy the hypotheses above, and $f : W_1 \to W_0$ is a finite branched cover of degree $d \geq 2$.

\begin{definition}
We define the \emph{repellor} $X$ of the system $(f, W_0, W_1)$ as
$$X := \bigcap_{n = 0}^\infty f^{-n}(W_1).$$
\end{definition}

By definition, $X = f^{-1}(X)$. Note that $X$ is compact, by compactness of the closure of $W_1$.
Moreover, we define the \emph{post-branch set} as
$$P_f := X \cap \bigcup_{n > 0} V_{f^n}.$$
Note that we do \emph{not} take the closure of the post-branch set (differently from \cite{HP}).

\medskip
Ha\"issinsky and Pilgrim \cite[Section 2.2]{HP} give the following topological definition of expansion.
Let $\mathcal{U}_0$ be a finite cover of $X$ by connected, open subsets of $W_1$ whose intersection with $X$ is not empty.
For each $n$, we define $\mathcal{U}_n$ as the open cover whose elements are the connected components of $f^{-n}(U)$ where
$U$ belongs to $\mathcal{U}_0$. We shall call \emph{pullbacks} of $U$ the connected components of the preimages $f^{-n}(U)$.

\begin{definition} \label{D:expansion}
A system $(f, W_0, W_1)$ satisfies the \textup{[Expansion]} axiom if there exists a finite cover $\mathcal{U}_0$ of $X$ such that 
the following holds:
for any open cover $\mathcal{Y}$ of $X$ by open subsets of $W_0$, there exists
$N$ such that for all $n \geq N$ each element of $\mathcal{U}_n$ is contained in some element of $\mathcal{Y}$.
We say that $\mathcal{U}_n$ is \textup{subordinated} to $\mathcal{Y}$.
\end{definition}

The following is useful. We denote as $\Comp_p U$ the connected component of $U$ containing $p$.

\begin{lemma}\label{L:repel}
Consider a system $(f, W_0, W_1)$ satisfying the \textup{[Expansion]} axiom with respect to a cover $\mathcal{U}$. 
Let $p\in X$ be a fixed point for $f$, where we allow $p$ to be critical. 
Then the following holds:
\begin{enumerate}
\item
There exist an open set $U \in \mathcal{U}$ containing $p$ and $N\in\N$ such that
$\Comp_p f^{-n}(U)$ lies in $U$ for all $n\ge N$.
\item 
Denoting $U_{n,p} := \Comp_p f^{-n}(U)$, we have
$$\bigcap_n U_{n,p} =\{p\}.$$
\item
(stronger than (1)) 
There exists $N$ such that for every open neighborhood $V$ of $p$ there exists $n_V \in \mathbb{N}$ such that 
the open set $U = U_{n_V, p}$ satisfies $U \subseteq V$ and 
$$\Comp_p f^{-n}(U) \subset U \;\;\;\; \hbox{for all}\;\; n\ge N.$$
\item
Moreover, for $U=U_{n_V+N,p}$ we obtain
$\Comp_p f^{-n}(U) \subset U_{n_V,p}\subset V$ for all $n\ge 0$. 
\end{enumerate}
If $p$ is periodic for $f$ of period $k$, all claims still hold after replacing $f$ by $f^k$.
\end{lemma}

\noindent In other words, $(2)$ says that each periodic orbit in $X$ is repelling (maybe not exponentially). 

\smallskip

If $W_0$ is equipped with a metric, \textup{[Expansion]} implies that the diameter of the cover $\mathcal{U}_n$ tends uniformly to zero as $n \to \infty$. The proof of Lemma \ref{L:repel} becomes immediate, since
 in Definition \ref{D:expansion} of \textup{[Expansion]} we can
take covers by discs of arbitrarily small diameters. In a purely topological situation the proof is less obvious and illustrates the adequacy of the \textup{[Expansion]} axiom.

First we provide the following.

\begin{lemma}\label{L:T3}
Let $(f, W_0, W_1)$ be a system satisfying the \textup{[Expansion]} axiom, with a finite cover $\mathcal{Y}$ of $X$ by open sets in $W_0$, with pullbacks $\mathcal{Y}_n, n=0,1,...$.
Then, given $p\in X$ and an element $Y(p)$ of $\mathcal{Y}_0$ containing $p$, we can modify $\mathcal{Y}$ to a new cover, 
still satisfying [Expansion], which contains $Y(p)$ and for which $Y(p)$ is the unique element containing $p$.
As a consequence, in the modified cover for every $n$ there is a unique element $Y_{n,p}\in\mathcal{Y}_n$ containing $p$.
\end{lemma}

\begin{proof} Choose an arbitrary $Y\in \mathcal{Y}_0$ containing $p$ and
subtract from all other $Y'\ni p$ a compact $K\subset Y$ containing $p$ in its interior, which exists by the $T_3$ property of our topology on $W_0$.
\end{proof}

\begin{proof}[Proof of Lemma \ref{L:repel}]
Let us apply Lemma \ref{L:T3} by taking as $\mathcal{Y}$ the cover $\mathcal{U}$ given by the \textup[Expansion] axiom.
Correct it to $\mathcal{U}'$ so that it has an element $U(p)$ which is unchanged and contains $p$.  
Then, there exists $N\in\N$ such that for all $n\ge N$ the old $\mathcal{U}_n$ is subordinated to $\mathcal{U}'$, 
with the pullback $U_{n,p}$ of $U(p)$ contained in a unique $U'(p)\in \mathcal{U}'$, hence in $U(p)\in \mathcal{U}$.
This proves (1). 

\smallskip

To prove (2), consider a basis $\mathcal{B}_p$ of our topology on $W_0$ at $p$. For each $Y\in \mathcal{B}_p$, take an arbitrary cover $\mathcal{Y}_Y$ of $X$ which includes $Y$, corrected by Lemma \ref{L:T3}.
Next, by [Expansion], for each $Y\in \mathcal{B}_p$ there is $n=n_Y$ so that
$\mathcal{U}_n$ is subordinated to $\mathcal{Y}_Y$ according to Definition \ref{D:expansion}. By uniqueness,
$U_{n,p} \subset Y$.
So the family $\{U_{n_Y,p}: Y\in \mathcal{B}_p\}$  also constitutes a basis at $p$,
hence (2) holds since our topology is Hausdorff. 

\smallskip

(3) Fixed $N$ as in (1), for an arbitrary $V$ choose $n_V$ as above.  We find $n_V$ such that $U_V:=U_{n_V,p}\subset V$. Then  taking adequate pullbacks of the sets in (1) we obtain
$\Comp_p f^{-n}(U_V) \subset U_V$ for all $n\ge N$.

\smallskip

Finally, writing $n=N+k$ for $n \ge N$ we can write
$\Comp_p f^{-k}(U_{n_V+N,p}) \subset U_{n_V,p}$ for all $k\ge 0$, yielding (4).
\end{proof}

By taking in the above proof an arbitrary $p\in X$ we prove the following fact, useful e.g. in the Proof of Theorem \ref{T:exp-contr}. 

\begin{lemma}\label{L:basis}
If a system $(f, W_0, W_1)$ satisfies the \textup{[Expansion]} axiom with a finite cover $\mathcal{U}$,
then the family $\mathcal{U}_n, n=0,1,2,...$ constitutes a basis of the topology on $W_0$ at all points of $X$.
Moreover, a basis is constituted by the family $\mathcal{U}_{Mn}$ for every integer $M\ge 1$.
\end{lemma}

\begin{proof}
As above, for every $p\in X$ and open $V\ni p$ there are $U_{n,p}\in \mathcal{U}_n$ in $V$ for all $n\ge n_V$, in particular for $n$ multiple of $M$.
\end{proof}

Note that the above Lemma \ref{L:basis} remains true if we replace $Mn$ by any sequence $n_i \to \infty$.

\begin{definition}
A system $(f, W_0, W_1)$ is said to satisfy the \textup{[Irreducibility]} axiom, see \cite{HP}, if $f$ restricted
to its repellor $X$ satisfies the \emph{locally eventually onto (leo)} property, namely for any $x \in X$ and any neighborhood $W$ of $x$ there is some $n$ with $f^n(W) = X$.
\end{definition}



\subsection{Weakly coarse expanding systems}

Let us now give a definition of weakly coarse expanding system.

\begin{definition}
A system $(f, W_0, W_1)$ is \emph{weakly coarse expanding} if:
\begin{enumerate}
\item it satisfies the \textup{[Expansion]} axiom;
\item it satisfies the \textup{[Irreducibility]} axiom;
\item the branch set $B_f$ is finite;
\item the repellor $X$ is not a single point.
\end{enumerate}
\end{definition}

Note that if $W_0 \subseteq S^2$ is an open subset of the $2$-sphere, then the set $B_f$ is always finite (see \cite{Why}).
Let us note that the last assumption implies that $X$ is uncountable and the topological entropy of $f$ on $X$ is positive.

We shall also use the following lemma (\cite[Proposition 2.4.1]{HP}), 
which can be proven using the fact that disjoint compact sets can be separated by disjoint open ones, 
a consequence of the compactness of $\overline{W_1}$ and the Hausdorff property. 

\begin{lemma}\label{L:chain}
Let $(f, W_0, W_1)$ be a weakly coarse expanding system with respect to the cover $\mathcal{U}$.
Then, there exists $N$ such that for any $U_1, U_2 \in \mathcal{U}_N$ with $U_1 \cap U_2 \neq \emptyset$, there exists
$U \in \mathcal{U}$ such that $U_1 \cup U_2 \subseteq U$.
\end{lemma}

\subsection{Exponentially contracting metrics}

A very important property of coarse expanding systems is that we can find a metric so that preimages shrink exponentially fast.

\begin{theorem} \label{T:exp-contr}
Suppose $f : W_1 \to W_0$ is a finite branched cover and axiom \textup{[Expansion]} holds.
Then there exist a metric $\rho$ on $X$ compatible with the topology and constants $C > 0$, $\theta < 1$ such that
for all $n \geq 0$
$$\sup_{U \in \mathcal{U}_n} \textup{ diam}_\rho (U) \leq C \theta^n.$$
\end{theorem}

We call a metric $\rho$ which satisfies the above property an \emph{exponentially contracting} metric. 
Observe that, as the metric $\rho$ is only defined on $X$, here and later $\textup{ diam}_\rho (U)$ means the diameter of $U \cap X$ (as in \cite{HP}).

An important example (with additional properties) are the \emph{visual metrics} constructed in \cite[Chapter 3]{HP} and \cite[Chapter 8]{BM}. Our notion of exponentially contracting metric is more general than the notion of visual metric, as we do not require any lower bound on the diameters of the elements of $\mathcal{U}_n$.
Note that by compactness of $X$ any metric $\rho$ on $X$ which induces this topology is complete.

The above theorem is essentially [\cite{HP}, Theorem 3.2.5]. We will, however, give a complete proof below, using a different method. We use Frink's metrization lemma \cite{frink}, in the following form.

\begin{lemma}[Frink's metrization lemma, \cite{PU}, Lemma 4.6.2]\label{L:Frink-lemma}
Let $X$ be a topological space, and let $(\Omega_n)_{n \geq 0}$ be a sequence of open neighborhoods of the diagonal $\Delta \subseteq X \times X$, such that
\begin{itemize}
\item[(a)]
$$\Omega_0 = X \times X$$
\item[(b)]
$$\bigcap_{n = 0}^\infty \Omega_n = \Delta$$
where $\Delta$ is the diagonal in $X \times X$.
\item[(c)]
For any $n \geq 1$,
$$\Omega_n \circ \Omega_n \circ \Omega_n \subseteq \Omega_{n-1}$$
where $\circ$ is the composition in the sense of relations: i.e., $R \circ S = \{ (x, y) \in X \times X \ : \ \exists z \in X  \textup{ s.t. } (x,z) \in R \textup{ and }(z,y) \in S \}$.
\end{itemize}
Then there exists a metric $\rho$ on $X$, compatible with the topology, such that
$$\Omega_n \subseteq \{ (x,y) \in X \times X \ : \ \rho(x,y) < 2^{-n} \} \subseteq \Omega_{n-1}$$
for any $n \geq 1$.
\end{lemma}

Before we prove Theorem \ref{T:exp-contr}, we prove its instructive 
metric version, assuming 
there is an adequate metric on $X$; 
this holds if 
the topology on $X$ has a countable basis 
\cite[Theorem 4.2.8]{Engelking}.

\begin{proposition} \label{P:Frink-exp}
Let $f : W_1 \to W_0$ be a finite branched cover with repellor $X$, let $\mathcal{U}$ be a finite cover of $X$ by open subsets of $W_1$ and let $\rho$ be a metric on $X$ such that
$$\lim_{n \to \infty} \sup\{ \textup{diam}_\rho (U) \ : \ U \in \mathcal{U}_n \} = 0.$$
Then there exist a metric $\rho'$ on $X$, which induces the same topology as $\rho$, and constants $C > 0$, $\theta < 1$ such that
$$\textup{diam}_{\rho'}(U) \leq C \theta^n$$
for any $n \geq 0$ and any $U \in \mathcal{U}_n$.
\end{proposition}

\begin{proof}
The proof follows from Lemma \ref{L:Frink-lemma}.
Let us define $V_n$ as the set of pairs $(x, y) \in X \times X$ such that there exists an element $U$ of $\mathcal{U}_n$
which contains both $x$ and $y$.
Let $\eta > 0$ be the Lebesgue number of $\mathcal{U}_0$ with respect to $\rho$.
By hypothesis there exists an integer $M \geq 1$ such that
$$\sup \{ \textup{diam}_\rho(U)  \  : \   U \in \mathcal{U}_l \} < \eta/3$$
for any $l \geq M$.
Now, let us define $\Omega_0 := X \times X$ and $\Omega_n := \bigcup_{k = 0}^{M-1} V_{Mn + k}$ for $n \geq 1$.
We need to check the hypotheses of Lemma \ref{L:Frink-lemma}.
Equation (a) is trivially true, as is the inclusion $\Delta \subseteq \bigcap_{n = 0}^\infty \Omega_n$ in (b). To prove the other inclusion, suppose that $(x, y) \in \Omega_n$ for any $n \geq 0$. Then for any $n$ there exists a connected set $\Gamma_n \in \mathcal{U}_{M n + k}$, 
with $0 \leq k \leq M-1$, which contains $x, y$.
Hence
$$\rho(x, y) \leq \sup \{ \textup{diam}_\rho(U)  \ : \ U \in \mathcal{U}_{l}, l \geq Mn \} \to 0$$
as $n \to \infty$. Thus $x = y$, as claimed. 

\smallskip

(Another way is to define $\Omega_n :=  V_{Mn}$ and refer to Lemma \ref{L:basis}.)

\smallskip

To prove (c), suppose to have sets $\Gamma_i \in \mathcal{U}_{M+k_i}$ for $i = 1, 2, 3$ with $0 \leq k_i < M$
and
$\Gamma_i \cap \Gamma_{i+1} \neq \emptyset$ for $i = 1, 2$.
Then $\textup{diam}_{\rho}(\Gamma_1 \cup \Gamma_2 \cup \Gamma_3) < \eta$, hence by definition of
Lebesgue number for the cover $\mathcal{U}_0\cap X$ of $X$,
there exists $\Gamma \in \mathcal{U}_0$ such that
$(\Gamma_1 \cup \Gamma_2 \cup \Gamma_3)\cap X \subseteq \Gamma$.
Hence
$$\Omega_1 \circ \Omega_1 \circ \Omega_1 \subseteq \Omega_0$$
and by taking $f^{M(n-1)}$-preimages,
$$\Omega_{n} \circ \Omega_{n} \circ \Omega_{n} \subseteq \Omega_{n-1}$$
for any $n \geq 1$, proving (c).

Thus, we can apply Lemma \ref{L:Frink-lemma}, obtaining that there exists a metric $\rho'$ on $X$,
which induces the same topology as $\rho$, such that
\begin{equation} \label{E:frink-inclusion}
\Omega_n \subseteq \{ (x,y) \in X \times X \ : \ \rho'(x,y) < 2^{-n} \} \subseteq \Omega_{n-1}
\end{equation}
for any $n \geq 1$.
Thus for any $U \in \mathcal{U}_{n}$ we have
$$\textup{diam}_{\rho'} (U) < 2^{-\lfloor \frac{n}{M} \rfloor}$$
for any $n \geq 0$, as claimed.
Indeed, we have proved that for any $U \in \mathcal{U}_{n}$ and for any $n \geq 0$ we have $\textup{diam}_{\rho'}(U) \leq C \theta^n$, with $C=2>0$ and $\theta = 2^{-1/M}<1$.
\end{proof}

Now let us prove Theorem \ref{T:exp-contr} using a similar procedure.

\begin{proof}[Proof of Theorem \ref{T:exp-contr}]
As in the proof of Proposition \ref{P:Frink-exp}, it is sufficient to check the hypotheses of Frink's lemma.

By Lemma \ref{L:chain}, the following holds:
there exists $N$ such that for all $U_1',U_2' \subset
\mathcal{U}_{N}$, if $U_1'\cap U_2'\not=\emptyset$ then there exists $U\in \mathcal{U}_0$ with
$U_1'\cup U_2'\subset U$.

Next,  given three sets $U_1', U_2', U_3'\in
\mathcal{U}_{2N}$ with $U_i'\cap U_{i+1}'\not=\emptyset$ for $i=1,2$,
we find
$U''_1\in \mathcal{U}_N$ containing $U'_1\cup U'_2$ and
$U''_2\in \mathcal{U}_N$ containing $U'_2\cup U'_3$ (using the previous fact for two sets and applying
$f^{-N}$). Finally we find $U\in\mathcal{U}_0$ with $U''_1 \cup U''_2 \subseteq U$.
Thus we obtain condition (c) of Frink's lemma for $\Omega_0 := X \times X$ and
$\Omega_n := \{ (x, y) \in X \times X \ : \ \exists U \in \mathcal{U}_{2 n N} \textup{ s.t. }x, y \in U \}$ for $n \geq 1$.
Condition (b) is a consequence of Lemma \ref{L:basis}.

The rest of the argument follows as in the proof of Proposition  \ref{P:Frink-exp}.
\end{proof}

\subsection{Equilibrium states}

Let us now recall some basic notions from ergodic theory. For more details, see e.g. \cite{PU}.

Let $f : X \to X$ be a continuous map of a compact metric space. A probability measure $\mu$ on $X$ is $f$\emph{-invariant} if $f_\star \mu = \mu$,
and we let $M(f)$ be the set of $f$-invariant probability measures on $X$.
We denote as $h_\mu(f)$ the metric entropy of $f$ with respect to $\mu$.

Now, consider a continuous function $\varphi : X \to \mathbb{R}$, which we call a \emph{potential}.
The \emph{topological pressure} of $f$ with potential $\varphi$ is defined as
$$P_{top}(\varphi) := \sup_{\mu \in M(f)}  \left\{ h_\mu(f) + \int_X \varphi \ d\mu  \right\}.$$
Note that $P_{top}(\varphi)$ may also be defined topologically (\cite{PU}, Section 3.2).
An $f$-invariant probability measure $\mu$ on $X$ is an \emph{equilibrium state} for $\varphi$ if it realizes the supremum, namely if
$$P_{top}(\varphi) = h_\mu(f) + \int_X \varphi \ d\mu.$$

In the following, we will consider a weakly expanding system $f : W_1 \to W_0$, and study the equilibrium states on its repellor $X$.

\subsection{The geometric coding tree}

Let us now construct a symbolic coding for a coarse expanding dynamical system. We start by proving the path lifting property.

\begin{lemma} \label{L:lift}
Let $f : Y \to Z$ be a finite branched cover, and let $\gamma$ be a continuous arc in $Z$ which is disjoint from the set of branch values $V_f$.
Let $x = \gamma(0)$, and $\widetilde{x}$ such that $f(\widetilde{x}) = x$. Then there exists a continuous arc $\widetilde{\gamma}$ in $Y$ such that $f(\widetilde{\gamma}) = \gamma$ and $\widetilde{\gamma}(0) = \widetilde{x}$.
\end{lemma}

\begin{proof}
Let $p \in \gamma$, and $\widetilde{p}$ a preimage of $p$. Then since $\gamma$ is disjoint from $V_f$, the local degree of $f$ at $\widetilde{p}$ is $1$.
Hence, there exists a neighborhood $U$ of $\widetilde{p}$ such that $f : U \to f(U)$ is injective, open, and closed, hence a homeomorphism; moreover, $f(U)$ is open.
Hence, one can lift $\gamma \cap f(U)$ to an arc $f^{-1}(\gamma) \cap U$ which contains $\widetilde{p}$. Thus, any point $p$ in $\gamma$ has a neighboorhood
$U_p$ over which $\gamma$ can be lifted, and moreover such that $f^{-1}(U_p)$ is the union of $d$ disjoint open sets which are homeomorphic to $U_p$; since $\gamma$ is compact, this implies that  the entire $\gamma$ can be lifted.
\end{proof}

The key point in our approach is that one constructs a semiconjugacy of a weakly coarse expanding system of degree $d$ to the shift map on $d$ symbols.

Let $\Sigma := \{1, \dots, d\}^\mathbb{N}$ be the space of infinite sequences of $d$ symbols, and $\sigma : \Sigma \to \Sigma$ the left shift.
If $\eta := (\eta_1, \dots, \eta_n) \in \{1, \dots, d\}^{n}$ is a finite sequence, the \emph{cylinder} associated to $\eta$ is the set $C(\eta) := \{ (\epsilon_i) \in \Sigma \ : \ \epsilon_i = \eta_i \textup{ for all }1 \leq i \leq n \}$. The integer $n$ is called the \emph{depth} of the cylinder $C(\eta)$. Note that for any $n \geq 1$, the set $\Sigma$ is the disjoint union of $d^n$ cylinders of depth $n$, and all cylinders are both open and closed.
Finally, let us equip $\Sigma$ with the metric $\rho((\epsilon_i), (\epsilon'_i)) := 2^{- \inf \{ k \geq 0 \ : \ \epsilon_k \neq \epsilon'_k \}}$,
which we call the \emph{standard metric} (with the convention $2^{- \infty} = 0$).

\begin{proposition} \label{P:semiconjugacy}
Let $f : W_1 \to W_0$ be a weakly coarse expanding system of degree $d$. 
Then there exists a H\"older continuous semiconjugacy $\pi : \Sigma \to X$
such that $\pi \circ \sigma = f \circ \pi$.
Moreover, if $f$ is locally eventually onto, then $\pi$ is surjective.
\end{proposition}

\begin{proof}
The construction is based on the idea of ``geometric coding tree" as in \cite{Feliks}.
Namely, pick $w \in X \setminus P_f$, which exists since $X$ in uncountable, and let $w_1, \dots, w_d$ be all its preimages. 
They are in $X$ by backward invariance of $X$.
For each $i = 1, \dots, d$, choose a continuous path $\gamma_i$ in $W_0$ connecting $w$ and $w_i$ and avoiding $P_f$.
This exists since the space is strongly path connected and the set $P_f$ is countable.

For each sequence $\alpha = (i_1, i_2, \dots) \in \Sigma$, define $z_n(\alpha)$ by letting $z_0(\alpha) := w_{i_1}$ and $z_n(\alpha)$ for $n \geq 1$ inductively as follows.
Let $\gamma_n(\alpha)$ be a curve which is the branch of $f^{-(n-1)}(\gamma_{i_n})$ such that one of its ends is $z_{n-1}(\alpha)$.
Such lifts exist by Lemma \ref{L:lift}, since we chose the curves $\gamma_i$ to be disjoint from $P_f$. Then define $z_n(\alpha)$ as the other end of $\gamma_n(\alpha)$. 


Now, since $X$ is the intersection of the nested compact sets $K_n := f^{-n}(\overline{W}_1)$, there exists $n_0$ such that every 
curve $\gamma_{n_0}(\alpha)$ for $\alpha \in \Sigma$ is contained in the union $\bigcup_{U \in \mathcal{U}_0} U$.

Note that there are $d^{n_0}$ distinct curves of form $\gamma_{n_0}(\alpha)$, parameterized by finite sequences of length $n_0$:
let us denote them $\gamma_{n_0}(\alpha_i)$ with $i = 1, \dots, d^{n_0}$. 
For each $i = 1, \dots, d^{n_0}$, we write $f_i :[0,1] \to W_0$ to denote the continuous map whose image is 
$f_i([0,1]) = \gamma_{n_0}(\alpha_i) \subset W_0$.
Note that the preimage of $\mathcal{U}_0$ under $f_i$ is an open cover of $[0,1]$. Take a refinement of that cover by open intervals, call it $\mathcal{V}^i$. 
We then extract a finite subcover $\{V^i_j \in \mathcal{V}^i : 1 \leq j \leq k_i\}$ for some $k_i$. Set $k=\max\{k_i, i=1, \dots , d^{n_0} \}$.
Now fix $n$, let $F_i : [0,1] \to W_0$ be a lift of $f_i$ under $f^n$, and let $\Gamma_i = F_i([0,1])$. Note that for each $V \in \mathcal{V}$, its image under any lift $F_i$ as above is a connected subset of $f^{-n}(U)$ for some $U \in \mathcal{U}_0$.
Therefore for each $V \in \{V_j^i : 1 \leq j \leq k_i\}$ the image $F_i(V)$ is a subset of some element of $\mathcal{U}_n$.
Thus we have at most $k$ open arcs that cover $\Gamma_i$, each of them in an element of $\mathcal{U}_n$. 

Applying Lemma \ref{L:chain} repeatedly and pulling back by $f^n$, there exists $N$ such that for any 
$n \geq N$,
if $U(1), \dots, U(k)$ is a ``chain" of elements of $\mathcal{U}_{n}$ with 
$U(j) \cap U(j+1) \neq \emptyset$ for any $j = 0,\dots, k-1$, then there exists $U \in \mathcal{U}_{n-N}$ with $U(j) \cup  U(j+1) \subseteq U$.

Now, using the claim proved in the paragraph above, for any $n \geq N$ 
by the exponential contraction property
of Theorem \ref{T:exp-contr}, we get, summing distances along consecutive pairs in the chain,
\[
\rho(z_n(\alpha), z_{n-1}(\alpha)) \leq C k \theta^{n-N-n_0},
\]
and hence
$$\lim_{n \to \infty} z_n(\alpha)$$
exists, and we define $\pi(\alpha)$ as the limit. By construction the map $\pi$ satisfies $f  \circ \pi = \pi \circ \sigma$ and $\pi(\alpha) \in X$.
To prove the H\"older continuity, let $\alpha = (i_1, i_2,\dots)$ and $\beta = (j_1, j_2, \dots)$ be two sequences, and let
$r := \min \{ s \ : \ i_s \neq j_s \}$, so that $z_{r-1}(\alpha) = z_{r-1}(\beta)$ if $r > 1$. Hence
$$\rho( \pi(\alpha), \pi(\beta) ) \leq \rho(\pi(\alpha), z_{r-1}(\alpha)) + \rho(z_{r-1}(\beta), \pi(\beta)) \leq 2 C' \theta^r $$
with $C' = (k C \theta^{-N - n_0})/(1 - \theta)$, which is what we need, as the distance between $\alpha$ and $\beta$ in the symbolic space is $2^{-r}$.

\smallskip 

To prove the coding map $\pi$ is surjective, let $x \in X$. Since $f$ is locally eventually onto, for each $k > 0$ there exists $n_k$ such that
$f^{n_k}(B(x, \frac{1}{k})) \supseteq X$.
In particular, there exists $w_k \in X$ with $\rho(x, w_k) < \frac{1}{k}$ and such that $f^{n_k}(w_k) = w$.
This means, there exists a finite sequence $\alpha_k = (i_1^{(k)}, \dots, i_{n_k}^{(k)})$ such that $w_k = z_{n_k}(\alpha_k)$ and
$\lim_{k \to \infty} w_k = x$.
By compactness of the shift space, up to passing to a subsequence there exists an infinite sequence $\alpha = (i_1, \dots, i_n, \dots)$
such that $\alpha_k \to \alpha$.

That is, for each $m$ there exists $k$ such that $\alpha$ and $\alpha_k$ coincide for the first $m$ symbols and $n_k\ge m$, hence by the exponential contraction property
$$\rho(z_m(\alpha), z_{n_k}(\alpha_k)) \leq C' \theta^m$$
for some $\theta < 1$. Then
$$\lim_m z_m(\alpha) = \lim_k z_{n_k}(\alpha_k) = x$$
thus $\pi(\alpha) = x$, as claimed.
\end{proof}

\section{No entropy drop}

Let $f : W_1 \to W_0$ be a finite branched cover with repellor $X$, let $\rho$ be a metric on $X$, and let $\varphi: (X, \rho) \to \mathbb{R}$
be a H\"older continuous potential. We now show that if there are no periodic critical points, then the entropy of any equilibrium state on $\Sigma$ is the same as the entropy of its pushforward on $X$.

\medskip
Inspired by \cite[Lemma 4]{Feliks}, we give the following definition.
We say $x$ is an \emph{$\epsilon$-singular point} if $\rho(x, B_f \cap X) < \epsilon$, that is if it lies within distance $\epsilon$ of a branch point.
Let $(x_i)_{0 \leq i \leq n}$ be a finite orbit segment, i.e. a sequence of points such that $f(x_i) = x_{i+1}$. Then we say $i$ is an $\epsilon$-singular time if
$x_i$ is an $\epsilon$-singular point.

\begin{lemma} \label{L:bounded-singular}
Suppose that no critical point is periodic.
Then for any $0 < \zeta < 1$, there exists $\epsilon > 0$ such that for any $x \in X$ and any finite orbit segment $(x_i)_{0 \leq i \leq n}$ with $x_n  = x$ we have
$$\# \{0 \leq  i \leq n \ : \ i \textup{ is an } \textup{$\epsilon$-singular time} \} \leq \zeta n + p,$$
where $p$ is the number of critical points.
\end{lemma}

\begin{proof}
Let $p$ be the number of critical points. Given $\zeta < 1$, let us choose $k$ such that
$\frac{p}{k} < \zeta$. Fix a critical point $c$, and let $\beta := \rho(c, \{ f^n(c) \}_{1\leq n \leq k}) > 0$.
We denote as  $B(x, r)$ the open ball of radius $r$ and center $x$ for the metric $\rho$.

We readily see that there exists $\epsilon > 0$ such that if $y \in B(c, \epsilon)$, then $f^l(y) \notin B(c, \epsilon)$ for any $l \leq k$.
Indeed, by continuity of $f$, there exists $0 < \epsilon < \beta/2$ such that $f^l(B(c, \epsilon)) \subseteq B(f^l(c), \beta/2)$ for any $l \leq k$.
Then if $y \in B(c, \epsilon)$ then $f^l(y) \in B(f^l(c), \beta/2)$, which is disjoint from $B(c, \epsilon)$ since $\epsilon < \beta/2$.

Moreover, we further shrink $\epsilon$ so that if $x \in X$ and the ball $B(x,\epsilon)$ does not intersect the branch set $B_f$, then the 
restriction $f \vert_{B(x,\epsilon)}$ is injective. 


Now suppose $(x_i)$ is an orbit segment, and $i < j$ are two singular times such that $x_i$ and $x_j$ lie within distance $\eta$ of the same critical point $c$.
Then by the argument above $j - i \geq k$.
Thus, the number of singular times is at most $p \lceil\frac{n+1}{k}\rceil \leq p (\frac{n}{k} + 1) < \zeta n + p$, as desired.
\end{proof}

Let us denote as $\sigma : \Sigma \to \Sigma$ the shift map on the symbolic space and $f : W_1 \to W_0$ the topological branched cover,
with repellor $X$.
Moreover, let $\pi : \Sigma \to X$ denote the semiconjugacy from Proposition \ref{P:semiconjugacy}.

\begin{lemma} \label{L:h-exp}
Suppose that no critical point is periodic.
For any $x \in X$, denote as $S_{n, x}$ the number of cylinders of depth $n$ which intersect $\pi^{-1}(x)$.
Then
$$\lim_{n \to \infty} \frac{1}{n} \log \sup_{x \in X} S_{n, x} = 0.$$
\end{lemma}

\begin{proof}
We shall use the notation $z_n(\alpha)$ for the vertices of the geometric coding tree as in the proof of Proposition
\ref{P:semiconjugacy}.
First, by exponential convergence of backward orbits, there exist $C > 0$, $\lambda < 1$ such that
$$\rho(z_n(\alpha), \pi(\alpha)) \leq C \lambda^n$$
for any $\alpha \in \Sigma$ and any $n \geq 0$.


Fix $\zeta$ with $0 < \zeta < 1$, and let $\epsilon > 0$ such that Lemma \ref{L:bounded-singular} is satisfied for $\zeta$.
Then there exists $N$ such that $C \lambda^m \leq \epsilon$ for any $m \geq N$.

Let $G := \{ z_n(\beta) \ : \ \beta \in \pi^{-1}(x) \}$.
Consider the subtree of the geometric coding tree formed by the union of the paths which connect each element of $G$ with the root $w$.
We say an integer $0 \leq k \leq n - 1$ is a \emph{branching level} for $G$ if there are two elements $z \neq z'$ in $G$ with $f^{k+1}(z) = f^{k+1}(z')$ but $f^{k}(z) \neq f^{k}(z')$.
Suppose that $k \leq n - N$ is a branching level for $G$, and let $z = z_n(\beta)$ and $z' = z_n(\beta')$ be two distinct points in $G$ with $f^{k+1}(z) = f^{k+1}(z')$ but $f^{k}(z) \neq f^{k}(z')$.
Note that by definition we have $f^{k}(x) = \pi(\sigma^{k} \beta)$
as well as $f^{k}(z) = z_{n-k}(\sigma^{k} \beta)$.

Now, by the above exponential convergence, we have
$$\rho(f^{k}(z), f^{k}(x))  = \rho(z_{n-k}(\sigma^{k}\beta)), \pi(\sigma^{k} \beta))  \leq C \lambda^{n-k} < \epsilon$$
and similarly
$$\rho(f^{k}(z'), f^{k}(x)) \leq C \lambda^{n-k} < \epsilon$$
hence $f$ is not injective on the ball of radius $\epsilon$ around $f^{k}(x)$, thus $k$ is an $\epsilon$-singular time for $x$.

Thus, a branching level for $G$ either satisfies $n - N \leq k \leq n - 1$ or is an $\epsilon$-singular time for $x$.
By Lemma \ref{L:bounded-singular} the number of $\epsilon$-singular times in the forward orbit $(x, f(x), \dots, f^n(x))$ is at most $\zeta n + p$, thus
the number of branching levels in $G$ is at most $N + \zeta n + p$.
Hence, the number of elements in $G$ is bounded above by
$$S_{n, x} = \# G \leq d^{N + \zeta n + p}$$
and
$$\limsup_{n \to \infty} \frac{1}{n} \log \sup_{x \in X} S_{n, x} \leq \zeta \log d.$$
By taking $\zeta \to 0$ we obtain that $\limsup \leq 0$. Further, since $S_{n, x} \geq 1$ the $\liminf$ is $\geq 0$,
which implies the claim.

\end{proof}

\begin{lemma} \label{L:no-drop}
Let $\mu$ be a $\sigma$-invariant measure on the symbolic space $\Sigma$, and let $\nu = \pi_\star \mu$ be the pushforward measure on $X$. Then
$$h_\mu(\sigma) = h_\nu(f).$$
\end{lemma}

\begin{proof}
Let $\mathcal{A}^n$ denote the partition of $\Sigma$ into cylinders of depth $n$, let $\theta$ be the partition in preimages of points under $\pi$,
and for $\nu$-a.e. $x \in X$ let $\mu_x$ be the conditional measure on the fiber over $x$.
Let us consider the \emph{relative entropy} $h_\mu(\sigma \vert f)$ (see e.g. \cite{Wa}), which we recall
satisfies
$$h_\mu(\sigma) = h_\nu(f) + h_\mu(\sigma \vert f).$$
By definition of relative entropy, 
\begin{align*}
h_\mu(\sigma \vert f) & = \lim_{n \to \infty} \frac{1}{n} H_\mu( \mathcal{A}^n \vert \theta) \\
& = \lim_{n \to \infty} \frac{1}{n} \int_X d\nu(x) \sum_{a \in \mathcal{A}^n} - \mu_x(a) \log \mu_x(a)
\end{align*}
and, by the comparison between measure-theoretic and topological entropy, 
\begin{align*}
& \leq\limsup_{n \to \infty} \frac{1}{n} \int_X \log S_{n, x} \ d \nu(x) = 0,
\end{align*}
where the last claim follows from Lemma \ref{L:h-exp}.
\end{proof}

\section{Existence and uniqueness of equilibrium states}

We start by showing that one can ``lift" invariant measures.

\begin{lemma} \label{L:measure-lift}
Let $\sigma : \Sigma \to \Sigma$ and $\tau : X \to X$ be continuous maps, and let $\pi : \Sigma \to X$ be a continuous semiconjugacy,
i.e. so that $\pi \circ \sigma = \tau \circ \pi$. Then for any $\tau$-invariant probability measure $\mu$ on $X$,
there exists a probability measure $\widetilde{\mu}$ on $\Sigma$ which is $\sigma$-invariant
and such that $\pi_\star \widetilde{\mu} = \mu$.
\end{lemma}

\begin{proof}
We will use Riesz' extension theorem. Namely, the measure $\mu$ defines a positive functional on $C_0(\Sigma) := \{ f \in C(\Sigma) \ : \ \exists g \in C(X) \textup{ with }f = g \circ \pi \}$
by setting
$$\mu( g \circ \pi ) := \int g \ d \mu.$$
The functional is well-defined since $\pi$ is surjective. Now, we need to check that
$$C(\Sigma) = C_0(\Sigma) + P$$
where $P$ is the convex cone of non-negative functions. This is obvious since given $f \in C(\Sigma)$ we have
$$f = \inf f + (f - \inf f)$$
where $\inf f$ is constant, hence belongs to $C_0(\Sigma)$, and $f - \inf f \geq 0$.
Then by the Riesz' extension theorem (\cite[Theorem 5.5.8]{Simon}) there exists a positive functional on $C(\Sigma)$ which extends $\mu$.
By Riesz' representation theorem, this functional is represented by a measure $\nu$ on $\Sigma$ which has the property that
$$\int  f \circ \pi \ d \nu = \int f \ d \mu$$
for any $f \in C(X)$. Now, by taking a limit point of the sequence
$$\mu_n := \frac{1}{n} \sum_{j = 0}^{n-1} \sigma^j_\star \nu$$
we obtain a measure $\widetilde{\mu}$ on $\Sigma$ which is $\sigma$-invariant and satisfies $\pi_\star \widetilde{\mu} = \mu$.
\end{proof}

Then, we show that the pushforward of an equilibrium state on $\Sigma$ is an equilibrium state on $X$.

\begin{lemma} \label{L:pressure}
Let $\varphi : (X, \rho) \to \mathbb{R}$ be a H\"older continuous potential, and let $\widetilde{\mu}$ be an equilibrium state for $\widetilde{\varphi} = \varphi \circ \pi$. Then $\mu := \pi_\star \widetilde{\mu}$ is an equilibrium state for $\varphi$.
\end{lemma}

\begin{proof}
Since $\pi$ is a semiconjugacy,
\begin{align*}
P_{top}(\varphi) & \leq P_{top}(\varphi \circ \pi)
\intertext{and since $\widetilde{\mu}$ is an equilibrium state}
& = h_{\widetilde{\mu}}(\sigma) + \int \widetilde{\varphi} \ d \widetilde{\mu}
\intertext{and using Lemma \ref{L:no-drop} (no entropy drop)}
& = h_\mu(f) + \int \varphi \ d \mu \leq P_{top}(\varphi)
\intertext{where in the last step we used the variational principle. Hence }
P_{top}(\varphi) & = P_{top}(\varphi \circ \pi)
\end{align*}
and $\mu$ is an equilibrium state for $\varphi$.
\end{proof}




\begin{lemma} \label{L:lift-measure}
Let $\mu$ be an equilibrium state for $\varphi$. Then there exists a measure $\widetilde{\mu}$ which is an equilibrium state for $\widetilde{\varphi}$
and $\pi_\star \widetilde{\mu} = \mu$.
\end{lemma}

\begin{proof}
By the proof of Lemma \ref{L:pressure},
\begin{align*}
P_{top}(\varphi \circ \pi) & = P_{top}(\varphi)
\intertext{and, since $\mu$ is an equilibrium state, }
& = h_\mu(f) + \int \varphi  \ d\mu
\intertext{and, using that entropy increases by taking an extension and that $\pi_\star \widetilde{\mu} = \mu$,}
& \leq h_{\widetilde{\mu}}(\sigma) + \int \widetilde{\varphi}  \ d\widetilde{\mu}
\end{align*}
hence by the variational principle $\widetilde{\mu}$ is an equilibrium state.
\end{proof}

\begin{lemma}
For any H\"older continuous potential $\varphi : X \to \mathbb{R}$ there is a unique equilibrium state.
\end{lemma}

\begin{proof}
Let $\mu_1, \mu_2$ be two equilibrium states for $\varphi$ on $X$. Then the measures $\widetilde{\mu_i}$ for $i = 1, 2$
produced in the previous Lemma are equilibrium states for $\varphi \circ \pi$ on $\Sigma$, and we know that this is unique.
Hence $\mu_1 = \pi_\star \widetilde{\mu_1} = \pi_\star \widetilde{\mu_2} = \mu_2$.
\end{proof}

\begin{proof}[Proof of Theorem \ref{T:main} (1)-(5)]
Consider a H\"older continuous potential $\varphi : X \to \mathbb{R}$ and a H\"older continuous observable $\psi : X \to \mathbb{R}$, and let $\pi: \Sigma \to X$ be the semiconjugacy of Proposition \ref{P:semiconjugacy}. Then $\varphi \circ \pi$ and $\psi \circ \pi$ are H\"older continuous with respect to the metric on $\Sigma$.
Now, by the previous discussion there is a unique equilibrium state $\mu_\varphi$ on $X$ for the potential $\varphi$ and a unique equilibrium state $\nu_{\varphi \circ \pi}$ for the potential $\varphi \circ \pi$. Moreover, it is well-known that equilibrium states for H\"older potentials on the full shift space on $d$ symbols
satisfy the statistical laws (CLT, LIL, EDC, LD); for CLT and LIL, see \cite[Theorem 5.7.1]{PU}; for EDC, see \cite[Theorem 5.4.9]{PU}, and for LD see for example \cite{DK} and references therein. Hence the sequence $(\psi \circ \pi \circ \sigma^n)_{n \in \mathbb{N}}$ satisfies the statistical laws with respect to $\nu_{\varphi \circ \pi}$.
Since $\psi \circ \pi \circ \sigma^n = \psi \circ f^n \circ \pi$ and $\nu_{\varphi \circ \pi}$ pushes forward to $\mu_\varphi$, the sequence $(\psi \circ f^n)_{n \in \mathbb{N}}$ also satisfies CLT, LIL, EDC and LD with respect to $\mu_{\varphi}$.
\end{proof}

\section{Periodic critical points} \label{S:periodic-crit}

Let us now focus on the case where $f : W_1 \to W_0 \subseteq S^2$ is defined on an open subset of the $2$-sphere, but periodic critical points are allowed.
In this case, we can blow up the sphere along the preimages of the critical orbits and obtain a coarse
expanding map of a Sierpi\'nski carpet (or a subset thereof) without periodic critical points.
Our construction turns out to be the inverse of the construction used in \cite{GHMZ}.

The main result of this section is the following proposition, whose proof we will give in several steps.

\begin{proposition} \label{P:blowup}
Let $f : W_1 \to W_0 \subseteq S^2$ be a weakly coarse expanding map on an open subset of the $2$-sphere, and let $\rho$ be an exponentially contracting metric on $X$.
Then there exist a strongly path connected space $\widetilde{W}_0$ and a weakly coarse expanding
system $g :  \widetilde{W}_1 \to \widetilde{W}_0$ without periodic critical points, with repellor $Y$, and a metric $\rho'$ on $Y$ which is
exponentially contracting with respect to $g$, and there is a continuous map $\pi : \widetilde{W}_0 \to W_0$ such that $\pi \circ g = f \circ \pi$.
\end{proposition}

Let us start the proof of Proposition \ref{P:blowup} by obtaining a local model in a neighborhood of a fixed critical point.

\begin{lemma} \label{L:conj}
Let $f : W_1 \to W_0 \subseteq S^2$ be a weakly coarse expanding map, and let $p \in X$ be a fixed critical point. Then
for any $\lambda > 1$ there exist $d \in \mathbb{Z} \setminus \{0 \}$, a neighborhood $U$ of $p$ and a homeomorphism $h : \overline{\mathbb{D}} \to \overline{U}$
such that $f \circ h = h \circ g$ where $g : \overline{\mathbb{D}}_{\lambda^{-1}} \to \overline{\mathbb{D}}$ is defined as
$$g(r e^{i \theta}) := \lambda r e^{i d \theta}$$
for any $r \leq 1$, $\theta \in \mathbb{R}$.
\end{lemma}

\begin{remark}
\textup{Note that $d$ may be negative if $f$ is orientation reversing in a neighborhood of $p$. The absolute value $|d|$ is the local
degree of $f$ at $p$.}
\end{remark}

\begin{proof}
Let $B$ be an open topological disc in $S^2$ containing the branched $f$-fixed point $p$ 
and whose closure does not contain other critical points, small enough so that all connected components of $f^{-1}(B)$ which do not contain $p$ are disjoint from $B$.

\medskip

\textbf{Claim.} There exists a Jordan curve $\gamma$ in $B \setminus \{p\}$ such that a component of
$f^{-1}(\gamma)$ is a Jordan curve disjoint from $\gamma$ and separates $\gamma$ from $p$.

\begin{proof}[Proof of the Claim]


By [Expansion], using Lemma \ref{L:repel}, there exist $N > 0$ and an open set $U$, with $p \in U \subseteq B$, such that for $n\ge N$,
\; the closure of $U_n := \Comp_p f^{-n}(U)$ is contained in $U$.
If we replace $U$ by $U_N$, its pullbacks $U_k, k=0, 1, \dots$ are all in $B$, 
so their closures do not contain critical points except possibly $p$.


Replace finally $U_N$ by a smaller \emph{Jordan domain} $\hat U$ containing $p$, that is an open topological disc with its boundary being a Jordan curve.
Let $V:= \Comp_p \left( \bigcap_{n=0}^{N-1} \hat{U}_n \right)$, where $\Comp_p$ means the component containing $p$.
The pullbacks $\hat{U}_n$ of the Jordan domain $\hat U$ must be Jordan domains, since they do not contain other critical points, see \cite{Why}. 
Since every connected component of the  intersection of finitely many Jordan domains in the plane is a
Jordan domain (see e.g. \cite[Proposition 2.4]{CoKo}), then $V$ is a Jordan domain, that is $\gamma := \partial V$ is a Jordan curve.

Similarly 
$V_1:=\Comp_p \left( \bigcap_{n=0}^{N-1} \hat{U}_{n+1}\right)$ is also a Jordan domain, and 
$V_1$ is a pullback of $V$, since $f^{-1}(\hat {U}_{N-1})$ has only one component in $B$ hence we can change the order of $f^{-1}$ and $\Comp_p$. 
Moreover $V_1 \subseteq V$. Let $\gamma_1 := \partial V_1$, which is also a Jordan curve. It is a pullback of $\gamma$.

Now, if the boundaries of $V$ and $V_1$ are disjoint, the claim is satisfied.
Otherwise, 
we shall correct $V$ so that the closure of $V_1$ is contained in $V$.
To this end, denote $a:=\gamma \cap \gamma_1$.
If $f(a)\subset a$ then $a$ is non-escaping, contradicting the [Expansion] axiom (in particular, Lemma \ref{L:repel}).

Otherwise, we modify $\gamma$ as follows.
Note that $a$ and $f(a)$ are closed sets, $f(a) \subset \gamma$, and $\gamma \setminus a$ is the union
of countably many open arcs, hence $f(a) \setminus a$ is a (non-empty) subset of $\gamma \setminus a$.
Now, we consider a neighborhood of $f(a) \setminus a$ in $\gamma \setminus a$, and modify there a part of $\gamma$ by pushing it slightly into $V$,
so that it is still disjoint from $\gamma_1$; note that we do not move points in the closure of $f(a)\setminus a$ belonging to $a$.
This way, we obtain a new Jordan curve $\gamma^1$ which
satisfies $\gamma^1 \cap f^{-1}(\gamma^1) = a \cap f^{-1}(a)$.

If the latter intersection is not empty, we repeat this process by considering for any $n$ the set
$$a_n := \{ x \in a \ : \ f^{k}(x) \in a \textup{ for all }k = 1, \dots, n \},$$
and we modify in $\gamma^{n-1} \setminus a_{n-1}$
a neighborhood of $f(a_{n-1}) \setminus a_{n-1}$  to obtain a new curve $\gamma^n$ so that
$\gamma^n \cap f^{-1}(\gamma^n) = a_n$.

Now, we claim there exists $N$ such that $a_N = \emptyset$. Indeed, otherwise the set $a_\infty=\{x\in\gamma: f^n(x)\in a \ \forall n \geq 1 \}$
would be non-empty with $f(a_\infty)\subset a_\infty$, contradicting the [Expansion] axiom.

Thus, the inductive procedure stops after $N$ steps, and $\gamma^N$ is disjoint from the preimage $f^{-1}(\gamma^N)$,
completing the proof of the claim.

\end{proof}

Let us now complete the proof of the lemma. By induction, let us define as $\gamma_0$ the curve $\gamma$ given by the previous claim, and as
$\gamma_{n+1} := f^{-1}(\gamma_n) \cap U$, obtaining a sequence of disjoint, nested simple closed curves which contain $p$ in their interior.
By the [Expansion] property, the diameter of $\gamma_n$ converges to $0$. For each $n$, let us define as $R_n$ the annulus bounded by $\gamma_n$ and $\gamma_{n+1}$. Let us pick a point $p_0$ on $\gamma_0$, and
let us choose as $p_1 \in \gamma_1$ as one of the $d$ preimages of $p_0$. Let us pick a continuous arc $\alpha = \alpha_0$ which joins $p_0$
and $p_1$ and is contained in $R_0$. Then for each $n \geq 1$ let us define  as $\alpha_{n+1}$ the component of $f^{-1}(\alpha_n)$ which starts at $p_n$,
and let $p_{n+1}$ be other end of $\alpha_{n+1}$. Let $\beta := \bigcup_{n = 0}^\infty \alpha_n$, which is a continuous (open) arc joining $p_0$ to $p$.
Note that by construction, $f(\beta) \cap U \subseteq \beta$. Now, for each $n$ denote as $\beta_{n, j}$ for $0 \leq j \leq d^n-1$ the connected components
of $f^{-n}(\beta) \cap U$.

Let us now fix $\lambda > 1$, and consider the map $g : \mathbb{D} \to \mathbb{D}$ defined as $g(r e^{i \theta}) := \lambda r e^{i d \theta}$. Let $\gamma_0' = \partial \mathbb{D}$ and for each $n$ denote as $\gamma'_n := g^{-n}(\gamma_0')$. Moreover, let $\beta' := (0, 1] \subseteq \mathbb{D}$, and for each $n$ let $\beta'_{n, j}$ for $0 \leq j \leq d^n - 1$ denote the components of $g^{-n}(\beta')$. Finally, let us define the annuli $R_n' \subseteq \mathbb{D}$ bounded by the curves $\gamma'_n$.

Now let us construct $h$. Define $h$ on $\gamma'_0$ as an arbitrary orientation-preserving homeomorphism
$h : \gamma'_0 \to \gamma_0$, such that $h(1)=p_0$.
Next, lift it to $h:\gamma_1'\to \gamma_1$ and proceed inductively lifting it to $h:\gamma'_n\to\gamma_n$ for any $n\in \mathbb{N}$, so that $f \circ h = h \circ g$ and $h(\lambda^{-n})=p_n$. This is possible because both $g$ and $f$ have the same degree $d\in \Z\setminus\{0\}$ (see \cite[Corollary 2.5.3]{Spanier}).

Then, extend $h$ continuously to a homeomorphism $h:\alpha'\to \alpha$, where $\alpha'=\alpha'_0:=[\lambda^{-1},1]\subset \overline{\D}$. Next, extend $h$ to a homeomorphism between the closed sets $R'_0$ and $R_0$
using as charts the Riemann mappings $\Phi':\D\to R'_0\setminus\alpha'_0$ and
$\Phi:\D\to R_0\setminus\alpha_0$ and their extensions to homeomorphisms (local, because two arcs are glued together to  $\alpha$)
$\hat\Phi'$ and $\hat\Phi$ on $\overline{\D}$, which exist by Carath\'eodory's theorem. Extend
$\hat\Phi^{-1} \circ h \circ \hat\Phi'$ from $\partial \D$ to $h'$ on
$\overline{\D}$ radially and next define the extended $h$ by $\hat\Phi \circ h' \circ (\hat\Phi')^{-1}$.

Finally,  define inductively $h:R'_n\to R_n$ for $n=1,2,...$ as lifts of $h:R'_0\to R_0$, extending continuously the map $h$
already defined on $\partial R'_n$. We complete the construction by setting $h(0) := p$.

Note that by this construction for any $n, j$ we have $h(\beta'_{n,j})=\beta_{n,j}$ (if for each $n$ they are
indexed by $j$'s in the same, say ``geometrical", order).
Therefore, $h$ maps the ``rectangles" bounded by $\gamma'_n, \gamma'_{n+1}, \beta'_{n,j}, \beta'_{n,j+1}$
to the ``rectangles" bounded by $\gamma_n, \gamma_{n+1}, \beta_{n,j}, \beta_{n,j+1}$.
\end{proof}

\begin{proof}[Proof of Proposition \ref{P:blowup}]
By Lemma \ref{L:conj}, we can topologically conjugate the map $f$ on a neighborhood of each periodic critical point to a map which linearly expands radial distances. Then using the linear model we can blow up each point in the periodic critical orbit to a circle, and define the map on each circle as multiplication by the local degree. Moreover, we also need to blow up points in the backward orbit of the critical point
and then we obtain a map $g : \widetilde{W}_1 \to \widetilde{W}_0 \subseteq \widetilde{S}$, where $\widetilde{S}$ is the Sierpi\'nski carpet obtained by replacing any point in the
backward orbit of a critical point by a circle. This space is still strongly path connected, so we can apply the techniques of the previous section.

\medskip
\noindent \textbf{Construction of the blowup.}
In order to discuss the details of this construction, let $\mathcal{C}$ denote the (finite) set of periodic critical points which lie in $X$, and let $\mathcal{E} = \bigcup_{n \geq 0}f^{-n}(\mathcal{C})$.
We can define a space $\widetilde{S}$ which is given by blowing up every point of $\mathcal{E}$ to a circle.
Namely, for each $q \in \mathcal{E}$ let $S_q$ be a copy of $S^1$. The space $\widetilde{S}$ is defined as a \emph{set} as
$$\widetilde{S} := (S^2 \setminus \mathcal{E}) \sqcup \bigsqcup_{q \in \mathcal{E}} S_q.$$
The topology on $\widetilde{S}$ will be defined shortly.
Note there is a natural projection map $\pi : \widetilde{S} \to S^2$ which sends each $S_q$ to $q$. Let $\widetilde{W}_0 := \pi^{-1}(W_0)$,
$\widetilde{W}_1 := \pi^{-1}(W_1)$.

\medskip
\noindent \textbf{Definition of $g$.}
Let us now extend $f$ to a map $g : \widetilde{W}_1 \to \widetilde{W}_0$. 
In order to do so, let us identify all $S_q$ for $q \in \mathcal{E}$ with $\mathbb{R}/\mathbb{Z}$,
and define $g : S_{q} \to S_{f(q)}$ as $g(\theta) := d \theta \mod 1$, where $d$ is the local (signed) degree of $f$ at $q$.
Finally, let us define $g := f$ on $\widetilde{W}_1 \setminus \bigcup_{q \in \mathcal{E}} S_q$.

\medskip
\noindent \textbf{Definition of the topology.}
Now, let us define a topology on $\widetilde{W}_0$ as follows.


Let us choose one element from any periodic orbit in $X$ containing critical points, and let us call $\mathcal{E}^*$ the union of such points. 
Moreover, 
for each $q \in \mathcal{E}$ let 
$m=m(q)$ be the minimal $m \geq 1$ such that $f^m(q) \in \mathcal{E}^*$. 
In particular for $p\in \mathcal{E}^*$ this is the minimal period.

Let now $p \in \mathcal{E}^*$.  Then by Lemma \ref{L:conj}, there exists a neighborhood $U_p$ of $p$ such that the map $f^m : U_p \to f^m(U_p)$ is topologically conjugate to $(r, \theta) \mapsto (\lambda r, d \theta \mod 1)$ for some $\lambda > 1$ and $|d| = \textup{deg}(f^m; p)$ is the local degree of $f^m$ at $p$.

We now fix some small values $r_0, \epsilon > 0$ and define for any $\theta_0 \in [0, 2 \pi)$ the set
$$
V_{p, \theta_0} := \pi^{-1}(h(\{  0 < r< r_0,  \theta_0 - \epsilon < \theta < \theta_0 + \epsilon \})) \cup \{ \theta \in S_p \ : \  \theta_0 - \epsilon < \theta < \theta_0 + \epsilon \}.$$

Suppose now $q \in \mathcal{E} \setminus \mathcal{E}^*$, let $m = m(q)$ and $p = f^{m}(q)$. 
By \cite[Theorem X.5.1]{Why2}, there exist a neighborhood $U_q$ of $q$ and a homeomorphism $h' : \mathbb{D} \to U_q$ such that 
$h^{-1} \circ f^m \circ h'(r e^{i \theta} ) = \lambda r e^{i d \theta}$, where $d$ is the local degree of $f^m$ at $q$ and $\lambda > 0$. Similarly as above, we define for any $\theta_0 \in [0, 2 \pi)$
$$
V_{q, \theta_0} := \pi^{-1}(h'(\{  0 < r< r_0,  \theta_0 - \epsilon < \theta < \theta_0 + \epsilon \})) \cup \{ \theta \in S_q \ : \  \theta_0 - \epsilon < \theta < \theta_0 + \epsilon \}.$$

\smallskip

\begin{remark} \label{R:2}
\textup{
Note that in the above construction the change of coordinates in the domain and range are different; 
thus, we do not claim the map to be topologically conjugate to its local model, as we do in Lemma \ref{L:conj}.
Moreover, here, just to define a basis of our topology at $S_q$ it is sufficient to consider the chart $h'=h_q$ 
for each $q$ separately. In fact, similarly one could define these charts inductively as a \emph{compatible system}, 
in the sense that $f\circ h_q=h_{f(q)}\circ F_q$ with e.g. $F_q(re^{i\theta}):=\lambda^{1/m(q)} r e^{i d \theta}$, where $d = \textup{deg}(f, q)$. 
As a consequence of Lemma \ref{L:conj}, this would also hold for $f^{m(p)}$ and $h_p = h_{f^{m(p)}} = h$ with $p \in \mathcal{E}^*$.
}\end{remark}

\smallskip

We now define the topology on $\widetilde{W}_0$ to be the topology generated by
$$\{ \pi^{-1}(U) \ : \ U \textup{ open in }W_0 \} \cup \{g^{-n}(V_{p, \theta_0}) \ : \  n \geq 0, p \in \mathcal{C}, \theta_0 \in [0, 2 \pi) \}.$$
Note that by the above construction it is immediate that the projection $\pi : \widetilde{W}_0 \to W_0$ is continuous.
Moreover, let us check that $g$ is also continuous with respect to this topology. First, if $\widetilde{U} = \pi^{-1}(U)$
with $U$ open in $W_0$, then $g^{-1}(\widetilde{U}) = \pi^{-1}(f^{-1}(U))$ is open in $\widetilde{W}_0$ since $f$ and $\pi$ are continuous.
If, on the other hand, $\widetilde{U} = g^{-n}(V_{p, \theta_0})$ for some $n, p, \theta_0$, then $g^{-1}(\widetilde{U}) = g^{-n-1}(V_{p, \theta_0})$ is also
open by construction, hence $g$ is continuous.

The topology is Hausdorff since for each $x,y \in \widetilde{W}_0$ such that $\pi(x)\not=\pi(y)$, there exist open disjoint sets $U_1, U_2$ in $W_0$ which separate $\pi(x)$ and $\pi(y)$, as the topology on $W_0$ is Hausdorff; then their preimages, which are open by our definition of topology on $\widetilde{W}_0$, separate $x$ and $y$. 
The only case to be checked is when $x, y$ belong to one $S_q$, but then they are separated by finite intersections of the sets $g^{-n}(V_{g^n(q),\theta_0})$.

Finally, note that $\widetilde{W}_0$ is metrizable, since it is $T_3$ and has a countable basis (\cite[Theorem 4.2.9]{Engelking}).
Moreover it can be homeomorphically embedded in $S^2$, see Appendix B.

\medskip
\noindent \textbf{Compactness.}
We now show that the closure of $\widetilde{W}_1$ is compact.

Given a family $\mathcal W$ of open sets covering $\widetilde S$, we can assume they belong to  a basis of the topology.
Those of the form  $\pi^{-1}(U)$ for $U\in\mathcal U$, the family of all open sets in $S^2$, will be called of type I, those intersecting $S_q$ of type II$_q$.

Given $q \in \mathcal{E}$,  by the compactness of $S_q$ we can choose a finite family $\mathcal W_q$ of sets of the cover of type II$_q$
which cover $S_q$. Their union contains an open neighborhood $V_q$ of $S_q$, of type I (since $\mathcal U$ separates points). Pick an open set $V'_q$ of type I so that $V_q\cup V'_q = \widetilde S$, and replace the cover $\mathcal W$ by a finer one, by intersecting all its elements with $V_q$ and $V'_q$.

Do it for all $q \in \mathcal{E}$ getting a resulting cover $\mathcal W'$. Consider the cover consisting of all the sets of type I in $\mathcal W$ and of the sets arising by replacing each finite family $\mathcal W_q$ by one set $V_q$: this is a cover of $\widetilde S$ by sets of type I, so we can find a finite subcover due to the compactness of $S^2$.

Replace back all $V_q$ in this finite cover  by the  finite family
$\mathcal W_q$. This gives a finite cover of $\widetilde S$ refining the original cover
$\mathcal W$. Compactness is proved.

\medskip
\noindent \textbf{Local connectivity.}  To show local connectivity of $\widetilde{W}_0$, we show that every point $x \in \widetilde{W}_0$ has a basis of (strongly) path connected neighborhoods. As usual, there are two cases.

\smallskip
\textbf{Case 1.}
Suppose $x \notin \pi^{-1}(\mathcal{E})$. Consider an open connected set $U \subseteq W_0 \subseteq S^2$ which contains $\pi(x)$, with the Euclidean topology. Then, since $\mathcal{E}$ is countable and
$U$ is a connected subset of the sphere, the set $U \setminus \mathcal{E}$ is (strongly) path connected.
Now, by construction the projection map $\pi$ is a homeomorphism between the set $\pi^{-1}(U) \setminus \bigcup_{q \in \mathcal{E}} S_q$
and $U \setminus \mathcal{E}$.
Thus, every point of $\widetilde{W}_0$ which is not in $\pi^{-1}(\mathcal{E})$ has a neighborhood basis which is strongly path connected.

\smallskip
\textbf{Case 2.} If $x$ belongs to some $S_p$ for some periodic critical point $p$, then consider $r_0$ as before, and for any $\epsilon > 0, \theta_0 \in [0, 2 \pi)$  the set
$$\mathcal{V} := \pi^{-1}(h(\{ 0 < r < r_0, |\theta - \theta_0| < \epsilon \})) \cup \{ \theta \in S_p \ : \ |\theta - \theta_0| < \epsilon \}.$$
Then, the set $\mathcal{V} \setminus \bigcup_{q \in \mathcal{E}\setminus \{p \}} S_q$ is homeomorphic to
$$\{ (r, \theta) \in \mathbb{R}^2  \ : \ 0 \leq r < r_0, \theta_0 - \epsilon < \theta < \theta_0 + \epsilon \} \setminus E$$
where $E$ is a countable set,  hence it is also strongly path connected.

Finally, if $x \in S_q$ for some $q \in f^{-n}(\mathcal{C})$, then one repeats the previous argument considering the connected components $\mathcal{V}'$
of $g^{-n}(\mathcal{V})$, where $\mathcal{V}$ is defined as above.

\medskip

\noindent \textbf{The repellor $Y$ and the [Irreducibility] axiom.} 
Now, we denote as $Y$ the repellor for $g : \widetilde{W}_1 \to \widetilde{W}_0$.

Note that by construction $Y = \pi^{-1}(X)$, so it contains $S_p$ for any $p \in \mathcal{E}$. 
Suppose for simplicity that $p$ is a fixed point for $f$.
Note now that $S_p$ is contained in the closure of the lift of $X\setminus  \{p\}$. 
To prove this, note that $X$ is an uncountable perfect set, so there is $x \in X$
arbitrarily close to $p$, and consider the set of all its $f^n$-preimages in $U = h(\D)$ (see Lemma \ref{L:conj}).
Since the repellor $X$ is backward invariant, these preimages belong to $X$, hence their lifts (points and circles) 
belong to $Y$ since $Y = \pi^{-1}(X)$.
But due to the degree of $f$ at $p$ being $d \geq 2$, they have arguments in the coordinates $h$ 
being all numbers of the form
$(\theta_n+ 2 \pi j)/d^n$ for a constant $0\le \theta_n<2\pi$ and integers $j:0\le j<d^n$, accumulating  on the whole $S_p$ as $n\to\infty$. 


In particular we obtain the locally eventually onto (leo) property of $g \vert_Y$. 
Indeed, if $W$ is an open neighbourhood  of $z\in S_p$ in $Y$ for $p$ a fixed point, 
then it contains a point $\tilde{x}$ whose projection by $\pi$ is $x\notin S_p$ as above. 
Therefore there exists $W'$ a neighbourhood of $\tilde{x}$ in $Y$ being the lift of an open set
$\pi(W')$ in $X$. By the leo property of $f$ on $X$ there is $n$ such that $f^n(\pi(W')) =X$. 
Hence $g^n(W)\supset g^n(W')=Y$. 
For $q$ in the grand orbit of a periodic critical point $p$ we use $g=f^{m(q)}$ 
and the definition of the basis of the topology at points in $S_q$ as the pullbacks 
of the basis at $S_p$. For $W$ in the basis of the topology at the points $z$ not in any $S_q$ 
we rely on the fact that, by construction, the basis at $z$ is the lift of the basis downstairs, at $\pi(z)$.

\medskip

\noindent 
\textbf{A contracting metric.} 
We claim that $(g, \widetilde{W}_0, \widetilde{W}_1)$ with an appropriate cover of $Y$ is a weakly coarse expanding system 
without periodic critical points.

Let us assume for simplicity (see also Remark \ref{R:4}) that there is only one periodic critical point $p$ in $X$ with $f(p) = p$, and let $d$ be the local degree of $f$ at $p$.

Let $X\subset W_1$ denote our $f$-invariant compact repellor. Let $\rho_1$ be a metric on $X$ compatible with the spherical topology given by the metric $\rho_e^X$ defined as the restriction to $X$ of the spherical metric $\rho_e$ on $S^2$.
Due to compactness both metrics are equivalent, that is for every $\epsilon>0$ there exists $\delta$ such that $\rho_1(x,y)<\delta$ implies $\rho_e^X(x,y)<\epsilon$, and vice versa.

\smallskip

Given a cover $\mathcal{U}$ and a metric $\rho$, we use the notation $\mesh_\rho(\mathcal{U}) := \sup_{U \in \mathcal{U}} \diam_\rho (U)$. 
Let  $\mathcal U_0$ be our cover of a neighborhood of $X$ and $\mathcal U_n$ the pullback covers
of a neighborhood of $X$ for $f^{-n}$. By [Expansion] we have 
$$\mesh_{\rho_1}(\mathcal U_n \cap X)\to 0$$
as $n\to\infty$, hence the same holds for $\rho_e^X$. We measure here the elements of $\mathcal U_n$ intersected with $X$.  Moreover, as $n\to\infty$ 
 \begin{equation} \label{E:shrink}
\mesh_{e}(\mathcal U_n)\to 0
\end{equation}
in the spherical metric $\rho_e$ in $S^2$ without intersecting with $X$, by Definition \ref{D:expansion}.

\smallskip

The Euclidean distance in $\overline{\mathbb{D}}$ being the domain of the chart $h$ in Lemma \ref{L:conj} will be denoted by $\rho_{e'}$. The same symbol will be used for the image of the metric on the range set $U$. The metrics $\rho_{e'}$ and $\rho_e$ are equivalent on $U$ by compactness.

\smallskip

By replacing $\mathcal{U}_0$ with some $\mathcal{U}_n$ (to be our new $\mathcal{U}_0$), assume
\begin{equation} \label{eq:2}
\textup{diam}_{e'}(U_n \cap h(\mathbb{D})) \leq \frac{1}{3} (1 - \lambda^{-2})
\end{equation}
for any $n \geq 0$ and any $U_n \in \mathcal{U}_n$, where $\lambda > 1$ is given by Lemma \ref{L:conj}.

We can assume that $U_0(p)=h(B(0,\lambda^{-N_0}))$, for an arbitrary $N_0$, is the set in $\mathcal U_0$ covering our branching fixed point $p$. This may be realized by first replacing $\mathcal U_0$ by $\mathcal U_n$ for an appropriate $n$ so that $U_n \in \mathcal U_n$, which covers $p$, is small enough; and next, replacing that $U_n$ by its superset $h(B(0,\lambda^{-N_0}))$.

For any connected $Z\subset h(\overline{\mathbb{D}}\setminus \{0\})$ denote by $\Delta(Z)$ the oscillation of the angle $\theta$ on $Z$.
Choose $\hat{r} > 0$ such that if $A \subseteq h(\mathbb{D})$ is a connected set with $A \nsubseteq U_0(p)$ and $\textup{diam}_{e} A < \hat{r}$, then $\Delta(A) < \pi$.
Now, by further replacing $\mathcal{U}_0$ with $\mathcal{U}_n$ (while keeping $U_0(p)$ as part of the cover), we can also assume
\begin{equation} \label{eq:1}
\textup{diam}_{e}(U_n) < \hat{r}
\end{equation}
for any $n \geq 0$ and any $U_n \in \mathcal{U}_n$ which is not a pullback of $U_0(p)$.
Finally, by removing any redundant open sets in $\mathcal{U}_0$, we may further assume that there are no other elements of $\mathcal{U}_0$
which are contained in $U_0(p)$.

\medskip

\noindent

\textbf{The cover and Frink's lemma.} 
 Denote $\lambda^{-N_0}$ by $r_\star$.
Let $\widetilde{\mathcal{U}}_0$ be the lift of $\mathcal U_0$ to a neighborhood of $Y$
after blowing up  a branching fixed point $p$ for $f$ and its grand orbit, to circles.
We add to $\widetilde{\mathcal{U}}_0$ two neighborhoods of arcs in the circle $S_p=\R/2\pi \Z$,
$V_0$ and $V_1$ with $0\le r < r_\star$ and $-\pi/4 < \theta<\pi+\pi/4$
and $\pi-\pi/4 < \theta < 2\pi +\pi/4$ respectively, in the polar coordinates of Lemma \ref{L:conj}. 
We replace by them the lift of the set $ U_0(p)\in \mathcal U_0$.
Denote this cover by $\widetilde{\mathcal W}_0$, and for each $n$ by $\widetilde{\mathcal{W}}_n$ the cover given by connected components
of the sets $g^{-n}(U)$ for any $U \in \widetilde{\mathcal{W}}_0$.

\medskip

\begin{remark} \label{R:3}
\textup{Notice that $\bigcup_n \widetilde{\mathcal{W}}_n$ provides a basis for the topology at all points of $Y$.
Indeed, consider any open $V\ni x$ where $x\in Y$.
If $\pi(x)\notin \mathcal{E}$ then by definition $V$ contains $V'=\pi^{-1}(U)$ where $U$ is open in $W_0$. 
By \textup[Expansion] for $f$, $U$ contains $U'\in\mathcal{U}_n$ for some $n$ (see Lemma \ref{L:basis}), hence $\pi^{-1}(U')\subset V$ and belongs to $\widetilde{\mathcal{U}}_n$.
Notice that we can replace the existing $n$ above by all $n\ge N$ for some $N=N_V$.
Analogously for $x\in S_q$ we easily choose $U \subseteq g^{-n}(V_i)$ with $x \in U \subseteq V$. \qed}
\end{remark}


\medskip
Given a set $A \subseteq \widetilde{W}_0$, let us denote
$A^{(2)} := (A \cap Y) \times (A \cap Y)$, and if $\mathcal{A}$ is a collection of sets, we denote 
$\mathcal{A}^{(2)} := \bigcup_{A \in \mathcal{A}} A^{(2)}$.
Moreover, if $\mathcal{A}, \mathcal{B}$ are collections of subsets of $\widetilde{W}_0$, let us denote
their 	``composition" as
$$\mathcal{A}^{(2)} \circ \mathcal{B}^{(2)} 
 := \{ (x,y) \in Y \times Y \ : \ \exists z \in Y \textup{ s.t. }(x,z) \in \mathcal{A}^{(2)}, (z,y) \in \mathcal{B}^{(2)} \}.$$

\begin{lemma} \label{L:Frink-Y}
For the cover $\widetilde{\mathcal W}_0$ there exists $N$ such that the following holds: 
\begin{itemize}
\item[(a)] There exists $N$ such that for any $n \geq 0$
$$
\widetilde{\mathcal W}^{(2)}_{n+N}  \circ \widetilde{\mathcal W}^{(2)}_{n+N}  \circ \widetilde{\mathcal W}^{(2)}_{n+N}  \subseteq \widetilde{\mathcal W}^{(2)}_{n}.
$$
\item[(b)] If $\mathbf{\Delta} = \{ (x, x) \in Y \times Y\}$ is the diagonal, then we have the intersection
$$\bigcap_{n = 0}^\infty \widetilde{\mathcal W}_{n N}^{(2)}=  \mathbf{\Delta}.$$
\end{itemize}
As a consequence, for the sequence $\Omega_n:= \widetilde{\mathcal{W}}^{(2)}_{nN}$ for $n \geq 1$, with 
$\Omega_0:= Y \times Y$, 
the hypotheses of Frink's lemma are satisfied.
\end{lemma}



\begin{proof}
(a) We start by proving the claim for $n = 0$; the general case will follow by taking preimages. Consider $W(i)\in \widetilde{\mathcal W}_N$, $i=1,2,3$ such that $W(i)$ intersects $W(i+1)$ in $Y$, for $i=1,2$.
We prove that if $N$ is large enough then there is $W_0\in \widetilde{\mathcal W}_0$ containing all three $W(i)$.

Let $\kappa_e$ be the Lebesgue number for $\mathcal U_0 \cap X$ and the metric $\rho_e^X$.
Choose $N$ so that $\diam_e \pi(W(i))\le\frac13 \kappa_e$ for $i=1,2,3$.
Additionally we assume that
$N$ is so large that $\diam_{e}\pi(W(i))< r'$ and
$\diam_{e'}\pi(W(i)\cap U_0(p)) < r'$ for a constant $r' \ll r_\star$.

\medskip

Consider $W :=\bigcup_{i=1,2,3} \pi(W(i))$. By definition of $\kappa_e$, the set $\pi(W) \cap X$ is
contained in some element of $\mathcal{U}_0$.

If $W$ lies in one $U_0\in \mathcal U_0$ different from $U_0(p)$, then we are done when we lift these objects to a neighborhood of $Y$.

So suppose $W$ lies in
$U_0(p)$. Fix $r' \ll  r_1 \ll r_\star$.
We consider two cases.

\medskip
1. $W \cap X$ is not contained in $h(B(0,r_1))$. Then since its diameter in $\rho_{e'}$ is at most $r'$, all $W(i)$ are contained either in $V_0$ or in $V_1$.

\medskip

2. $W\cap X \subset h(B(0,r_1))$.  Let $k>0$ be the largest integer such that $f^j(W)\subset U_0(p)$ for all $j=0, \dots, k$. 
Since $r_1 \ll r_\star$, the number $k$ is large (note that we do not intersect here $W$ with $X$).
There are two sub-cases.

\smallskip

2A. If $W(i)$ is a pullback of some $U$ in $\mathcal{U}_0 \setminus \{ U_0(p) \}$, then $k < N$,
no other elements of $\mathcal{U}_0$ are contained in $U_0(p)$.
Then, since $f^{k+1}(\pi(W(i)))$ belongs to $\mathcal{U}_{N-k-1}$ and by \eqref{eq:1}, we have
$$\Delta(f^{k+1}(\pi(W(i)))) < \pi.$$
In the case $g^k(W(i))=V_0$   or $V_1$, we obviously also have $\Delta(g^k(W(i))) < \pi$.

\smallskip

2B. Otherwise, $W(i)$ is a pullback of either $V_0$ or $V_1$
 for $g^{-\ell}$ with $\ell>0$. Then $\ell>N_0$ and by \eqref{eq:2}
the set $f^{k+N_0}(W)$ lies in the annulus
$$h(\{ \lambda^{-2} \leq r \leq 1 \})$$
hence by \eqref{eq:2} we obtain
$\Delta f^{k+N_0}(W)<\pi$.

Our $k$ is large, so $\Delta(W)\le \pi d^{-k}$, smaller than $\pi/2$.
In 2B. we have used the fact that each $\pi(W(i))$ is a pullback of $f^{k+N_0}(\pi(W(i)))\subset h(\overline{\mathbb{D}})$ for $f^{-(k+N_0)}$. (We would have troubles if we intersected with $X$ without invoking full pullbacks.)
 Hence all three $W(i)\cap X$ are entirely in $V_0$ or in $V_1$.
 This ends the proof of the claim for the case $n = 0$.

Finally, by taking preimages $\widetilde{\mathcal W}_{n+N}=g^{-n}(\widetilde{\mathcal W}_{N})$ in $\widetilde{\mathcal W}_n$ we obtain (a) 
for all $n \geq 0$.

\medskip

(b) It follows from the Hausdorff property of the topology on $\widetilde{W}_0$ and the fact that the covers $\widetilde{\mathcal{W}}_{n N}$ are a basis of the topology at $Y$, as established in Lemma \ref{L:basis} and Remark \ref{R:3}.





\end{proof}

Now, we can apply Frink's lemma (Lemma \ref{L:Frink-lemma}) to the sets 
$\Omega_0 = Y \times Y$ and $\Omega_n := \widetilde{\mathcal W}^{(2)}_{nN}$, for $n \geq 1$, obtaining a metric $\rho'$ on $Y$,
which is compatible with the topology introduced above, such that
$$\mesh_{\rho'} \left( \widetilde{\mathcal W}_{nN} \right) \leq 2^{-n}$$
for any $n \geq 0$. This shows that the metric $\rho'$ is exponentially contracting with respect to the cover $\widetilde{\mathcal W}_N$.

\begin{remark} \label{R:4}
\textup{If there are more than one periodic critical point, with period possibly higher than $1$, one extends the cover 
$\widetilde{\mathcal{U}}_0 = \pi^{-1}(\mathcal{U}_0)$ by adding two sets similar to $V_0, V_1$ for each $p \in \mathcal{E}^*$, and checks the hypotheses of Frink's lemma by applying the previous proof to $f^m$, where $m$ is the least common multiple of all periods of elements in $\mathcal{E}^*$. }
\end{remark}

\noindent \textbf{The [Expansion] axiom.} Finally, we prove that the cover $\widetilde{\mathcal W}_0$ satisfies [Expansion]. 

Let $\mathcal{V}$ be an arbitrary finite cover of $Y$. By Remark \ref{R:3}, 
there exists a cover of $Y$ by the sets
$U(x,n) \in \widetilde{\mathcal{W}}_n$, each being a subset of an element $V$ of $\mathcal{V}$ for $n\ge N(V,x)$.
Then, by \cite[Theorem 3.1.6]{Engelking}, together with compactness and the Hausdorff property, we can find, for each $V \in \mathcal{V}$, a compact set $V^*\subset V\cap Y$ so that the union of the $V^*$ covers $Y$. 

Let us now fix $V \in \mathcal{V}$. We shall prove that over all $x \in V^*$, the values $N(V,x)$ are uniformly bounded. Let  $V^c:= \overline{ \widetilde{W_1}} \setminus V$.
Recall that for each $q \in \mathcal{E}$ we denote as $m(q)$ the smallest integer $m \geq 1$ such that $f^m(q)$ 
is an element of $\mathcal{E}^*$.
 
\medskip
\textbf{Case 1.} 
Suppose that $\pi(V^*)\cap \pi(V^c)=\emptyset$. Since $\pi$ is continuous, these sets are compact, 
so their Euclidean distance in $S^2$ is positive, say $\delta_V>0$. So if $U\in \widetilde{\mathcal{U}}_n$
intersects both $V^*$ and $V^c$, the Euclidean diameter of $\pi(U)$ is at least $\delta_V$, so by \eqref{E:shrink} $n$ is bounded from above by some $N_V$, independent of $x\in V^*$. In other words $U(x,n)$ cannot intersect both $V^*$ and $V^c$ for $n>N_V$, hence is in $V$.

The same holds for components of $g^{-n}(V_i), i=1,2$, since their projections by $\pi$ are in 
$U_n\in \mathcal{U}_n$ by construction, with diameters tending to 0 by \eqref{E:shrink}. 

\medskip

\textbf{Case 2.} If $\pi(V^*)$ and  $\pi(V^c)$ intersect, then there exists $q\in\mathcal{E}$ such that both $V^*$ and $V^c$ intersect $S_q$. 
In such a case we have the: 

\smallskip
\textbf{Claim.} Among all $q$'s so that $V^*$ and $V^c$ intersect $S_q$, the set of possible $m(q)$ is bounded.

\smallskip
{\textit{Proof of the Claim.}} Otherwise, let $(q_n) \subseteq X$ be a sequence with $m(q_n) \to \infty$, and choose a convergent subsequence $q_n\to q^*$ using the metric $\rho_1$ on $X$. Let $x_n\in S_{q_n}\cap V^*$ and $y_n\in S_{q_n}\cap V^c$. 
If $q^*\notin\mathcal{E}$ then $x_n, y_n$ converge as $n\to\infty$ to the only preimage $\widetilde{q}^*$ of $q^*$ for $\pi$ because they enter a basis of neighborhoods of $\widetilde{q}^*$, which is the preimage of a basis at $q^*$. So $V^*$ and $V^c$ intersect, a contradiction.
If $q^*\in\mathcal{E}$ then let $U$ be an open neghborhood of $q^*$ given by Remark \ref{R:3}, and
let $\pi_{q^*}$ be the map from $\pi^{-1}(U)$ in $\widetilde{W}_0$ to $\hat U:=(U\setminus \{q^*\}) \sqcup S_{q^*}$ 
which blows down every circle except $S_{q^*}$.
The topology and metric on $\hat U$ are pullbacks under $f^{m(q^*)}$ of the Euclidean metric and topology from the chart $h$ on a neighborhood of $p=f^{m(q^*)}(q^*)$, extended to the boundary circles. 
Next choose from $\pi_{q^*}(x_n)=\pi_{q^*}(y_n)$ a subsequence convergent to a point $x$ in $S_{q^*}$.
By definition $\pi_{q^*}$ is continuous, so $\pi_{q^*}(V^*)$ and $\pi_{q^*}(V^c)$ are compact, hence they intersect at $x$. But $\pi_{q^*}^{-1}(x)$ is one point where therefore $V^*$ and $V^c$ intersect, a contradiction. \qed
  
\medskip
Since $m(q)$ is bounded, we need to consider only a finite number of $q$'s. 
In each $S_q$,   $V^*$ and $V^c$ are compact disjoint so for $n$ large the components of $g^{-n}(V_i)$ in $S_q$ have diameters too small to intersect both $V^*$ and $V^c$.
  
\
The conclusion is that $\widetilde{\mathcal{W}}_n$ is subordinated to $\mathcal{V}$ for all $n$ large enough. The \textup[Expansion] axiom has been verified, which completes the proof of Proposition \ref{P:blowup}. 

\end{proof}

Let us remark that the above proof of [Expansion] together with Theorem \ref{T:exp-contr} yields another proof of the 
existence of an exponentially contracting metric $\rho'$ on $Y$.

\subsection{Topologically H\"older functions} \label{S:top-holder}

Given a cover $\mathcal{V}$ of a topological space $X$ and a function $f : X \to \mathbb{R}$,
we define the \emph{total variation} of $f$ with respect to $\mathcal{V}$ as
$$\textup{var}_\mathcal{V}  f := \sup_{V \in \mathcal{V}} \sup_{x, y \in V} |f(x) - f(y)|.$$
Moreover, given a sequence $(\mathcal{V}_n)_{n \geq 0}$ of covers of a topological space $X$, a function $f : X \to \mathbb{R}$ is \emph{topologically H\"older} with respect to $(\mathcal{V}_n)$ if
there exist constants $C, \alpha > 0$ such that
$$\textup{var}_{\mathcal{V}_n} f \leq C e^{-n \alpha}$$
for any $n \geq 0$.

The definition is inspired by Bowen's definition of H\"older continuous potentials for the shift map (see \cite[Theorem 1.2]{bowen}).
Let us collect a few fundamental properties about these functions, whose proofs are straightforward.

\begin{lemma} \label{L:topHolder}
Topologically H\"older functions satisfy the following properties.
\begin{enumerate}
\item
Let $f : W_1 \to W_0$ be a finite branched cover with repellor $X$, let $\mathcal{U}$ be a cover of $W_1$, and
let $\rho$ be an exponentially contracting metric on $X$.
Then a function $\varphi : X \to \mathbb{R}$ which is H\"older continuous with respect to $\rho$ is topologically H\"older
with respect to the sequence of covers $\mathcal{U}_n$, whose elements are the connected components of $f^{-n}(\mathcal{U})$.
\item
If $\varphi : X \to \mathbb{R}$ is topologically H\"older with respect to $(\mathcal{U}_n)$ and $\pi : Y \to X$ is continuous,
then $\varphi \circ \pi$ is topologically H\"older with respect to $(\mathcal{V}_n) := (\pi^{-1} (\mathcal{U}_n))$.
\item
If $X = \Sigma$ is the full shift space and $(\mathcal{U}_n)$ are the partitions in cylinders of rank $n$, then $\varphi : X \to \mathbb{R}$
is topologically H\"older with respect to $(\mathcal{U}_n)$ if and only if it is H\"older continuous with the respect to the standard metric on $\Sigma$.
\end{enumerate}
\end{lemma}

\begin{proposition} \label{P:unique-blowup}
Let $f : W_1 \to W_0$ be a weakly coarse expanding system, with $W_0 \subseteq S^2$, and let $(X, \rho)$ be its repellor, with $\rho$ an exponentially contracting metric.
Then for any H\"older continuous potential $\varphi : (X, \rho) \to \mathbb{R}$ there exists a unique equilibrium state $\mu_X$ on $X$.
\end{proposition}

\begin{proof}
Let $\varphi : X \to \mathbb{R}$ be a H\"older continuous potential with respect to an exponentially contracting metric $\rho$.
By Proposition \ref{P:blowup}, there is a metric space $(Y, \rho')$ with a continuous map $g : Y \to Y$ such that $\rho'$
is exponentially contracting, and a continuous projection $\pi: Y \to X$
such that $\pi \circ g = f \circ \pi$. Moreover, by Proposition \ref{P:semiconjugacy}, there exists a coding map $\Pi : \Sigma \to Y$
with $\Pi \circ \sigma = g \circ \Pi$.

By Lemma \ref{L:topHolder}(1), the map $\varphi : X \to \mathbb{R}$ is topologically H\"older. Then, by Lemma \ref{L:topHolder} (2), the map $\varphi \circ \pi \circ \Pi$
is also topologically H\"older, hence by Lemma \ref{L:topHolder} (3) the map $\varphi \circ \pi \circ \Pi$ is also H\"older continuous with respect to the standard metric
on the shift space $\Sigma$.

Thus, by Bowen \cite[Theorem 1.22]{bowen}, there is a unique equilibrium state $\mu$ on $\Sigma$ for the potential $\varphi \circ \pi \circ \Pi$. This measure is ergodic with respect to the shift map, and  positive on non-empty open sets (as a consequence of the Gibbs property \cite[Theorem 1.2]{bowen}).

Notice that $\Pi_*(\mu)$ is $0$ on $S_p$ and its $g^n$-pre-images. Otherwise, by ergodicity it
would be supported on $S_p$  (if it charged other preimages of $S_p$
it would charge preimages under iterates, therefore being infinite).
So $\mu$ would be 0 on the open nonempty set $\Sigma\setminus \Pi^{-1}(S_p)$, a contradiction.

So $\pi\circ \Pi$ preserves the entropy of $\mu$ since $\Pi$ does (Lemma \ref{L:no-drop})
and $\pi$ does as it is a measurable isomorphism $\pi : Y \setminus \bigsqcup_{p \in \mathcal{E}} S_p \to X \setminus \mathcal{E}$.
Moreover, if we set $\mu_X:=(\pi \circ \Pi)_*(\mu)$, then $\int_\Sigma \varphi\circ \pi \circ \Pi \ d\mu = \int_X \varphi \ d \mu_X$
by definition.

Hence the projection preserves pressure; namely, by the variational principle
$$
P_{top}(f,\varphi) \ge h_{\mu_X}(f) + \int_X \varphi \ d\mu_X = h_\mu(\sigma) + \int_\Sigma \varphi\circ \pi \circ \Pi \ d\mu=
P_{top}(\sigma, \varphi\circ \pi \circ \Pi).
$$
Since the opposite inequality is true in general, following from definitions, this yields the equality $P_{top}(f,\varphi) =
P_{top}(\sigma, \varphi\circ \pi \circ \Pi)$.
In particular, we have proved that $\mu_X$  is an equilibrium state.

This equilibrium state is unique; indeed, any other equilibrium state would have by Lemma \ref{L:measure-lift} a lift to $\Sigma$,
and this would also be an equilibrium state, as the lift cannot decrease either the entropy nor the integral, contradicting the uniqueness on $\Sigma$.
\end{proof}

\begin{proof}[Proof of Theorem \ref{T:main-sphere} (1)-(5)]
By Proposition \ref{P:unique-blowup}, for any H\"older continuous potential $\varphi : X \to \mathbb{R}$ there is a unique equilibrium state $\mu_X$ on $X$, and this equilibrium state is obtained as the pushforward of the unique equilibrium state $\nu$ for the potential
$\varphi \circ \pi \circ \Pi$, where $\pi : Y \to X$ is the blowdown map and $\Pi : \Sigma \to Y$ is the coding map.
Then, as in the case without periodic critical points, for any H\"older continuous observable $\psi$ statistical laws for the sequence $(\psi \circ f^n)_{n \in \mathbb{N}}$ follow from the well-known statistical laws for the shift map with respect to the sequence $(\psi \circ \pi \circ \Pi \circ \sigma^n)_{n \in \mathbb{N}}$.
\end{proof}

\section{The cohomological equation}

Let us now see the proof of \eqref{itm:sigma0} and \eqref{itm:cohom} in Theorems \ref{T:main} and \ref{T:main-sphere}.

\begin{proposition}\label{prop:homology}
Let $\rho$ be a visual metric on $X$. If the measures $\mu_\varphi$, $\mu_\psi$ coincide then there exist a H\"older continuous function $u:X\to \mathbb R$ and a constant $K\in \mathbb R$ such that
\begin{equation} \label{E:cohom}
\varphi-\psi=u\circ f-u+K.
\end{equation}
\end{proposition}

Assume that $\mu_\varphi=\mu_\psi=\mu$. We know that
$\mu_\varphi=\pi_*\nu_{\varphi\circ \pi}$ and $\mu_\psi=\pi_*\nu_{\psi\circ \pi}$ where
$\nu_{\varphi\circ \pi}$, $\nu_{\psi\circ \pi}$ are, respectively, the equilibrium states for $\varphi\circ \pi$ and $\psi\circ\pi$.
Put $\eta=\varphi-\psi$. Adding constants to $\varphi, \psi$, we may assume that $\int\varphi \ d\mu=0$ and $\int\psi \ d\mu=0$, so that $\int\eta \ d\mu=0$.
Denote by $S_k\eta(x)$ the sum
$$S_k\eta(x)=\eta(x)+\dots+\eta(f^{k-1}(x)).$$

We can detect equality between invariant measures by looking at periodic points:

\begin{proposition} \label{P:periodic}
If for H\"older continuous $\varphi, \psi:X\to \mathbb R$ the equilibrium states satisfy $\mu_\varphi=\mu_\psi$
(denote them by $\mu$) and if
$\int \eta \ d\mu=0$ for $\eta=\psi-\varphi$,
then there exists $C>0$ such that
\begin{equation} \label{E:per-bounded}
|S_n \eta (x)|<C
\end{equation}
for all $x\in X$  and $n \in \mathbb{N}$.
\end{proposition}

\begin{proof}[Proof of Proposition~\ref{P:periodic}]
If $\mu_\varphi$ and $\mu_\psi$ coincide, then also $\nu_{\varphi\circ \pi}$ and $\nu_{\psi\circ\pi}$ coincide.
Indeed, the proof of Lemma \ref{L:lift-measure} shows that any lift of $\mu_\varphi$ is an equilibrium state for $\varphi \circ \pi$. Hence, $\nu_{\psi \circ \pi}$ is
an equilibrium state for $\varphi \circ \pi$, but equilibrium states are unique on the shift space, hence $\nu_{\psi \circ \pi} = \nu_{\varphi \circ \pi}$.
But then (see, e.g. \cite{bowen}, Theorem 1.28) there exists a
constant $C$ such that  for every $\omega \in \Sigma$ and every $n$,
$$|S_n(\eta\circ \pi)(\omega)|\le C.$$
Since the coding $\pi$ is onto, this also implies \eqref{E:per-bounded}.
\end{proof}

Note that $\pi^{-1}(p)$ is not necessarily finite, so it may not be a periodic orbit of $\sigma$.
So, we conclude the following:

\begin{corollary} \label{C:periodic-zero}
If $\mu_\varphi$, $\mu_\psi$ coincide and the integral $\int\eta d\mu=\int(\varphi-\psi)d\mu=0$ then for every periodic point $p$
with $f^k(p)=p$ we have that
\begin{equation}\label{sum_periodic_zero}
S_k\eta(p)= 0.
\end{equation}
\end{corollary}

\begin{proof}
Suppose $S_k\eta(p) \neq 0$.
Then for any $n$
$$S_{nk}\eta(p)=n\cdot S_k\eta(p),$$
so that the sequence  $(S_{nk} \eta)_{n \geq 0}$ is not bounded, a contradiction.
\end{proof}

Using the above characterization, we can solve the cohomological equation \eqref{E:cohom}.

\begin{proof}[Proof of Proposition~\ref{prop:homology}]
We can assume that the integral of $\eta=\varphi-\psi$ is equal to zero.  As in the classical proof (see \cite{bowen}, proof of Theorem 1.28), begin with choosing a point $x\in X$ with dense trajectory, and define the function $u$ on the set
$$S:=\{f^n(x)\}_{n\ge0}$$
by putting
$$u(x) := 0, \quad u(f^n(x)) :=\eta(x)+\dots +\eta(f^{n-1}(x))\quad \text{for}\quad  n>0.$$
Then, clearly, for every $z\in S$, $z=f^n(x)$
the equation
\begin{equation}\label{eq:homology}
u(f(z))-u(z)=\eta(z)=\varphi(z)-\psi(z)
\end{equation}
holds.

Now, it is enough to prove that the function $u: S \to \mathbb R$ is H\"older continuous, i.e.,  that there exist $C>0$, $\alpha>0$ such that
\begin{equation}\label{eq:holder}
|u(f^k(x)) - u(f^m (x))| <C \rho(f^m(x), f^k(x))^\alpha.
\end{equation}
Then  $u$ extends to a H\"older continuous function on $X$ and  the equation \eqref{eq:homology} extends to $X$.

To prove that $u$ is H\"older continuous, consider two points $z = f^k(x), y = f^m(x)$, $k<m$.
Let $n$ be the largest integer such that $y, z$ are in the same component of $\mathcal U_n$. Denote this component by $U_n$.

Let $l := m -k$ so $y = f^l(z)$, and for $r=1,\dots l$ denote, by induction  by $U_{r+n}$ the connected component of $f^{-1}(U_{r-1+n})$ containing the point $f^{m-r}(x)$. To conclude the proof, we need the following

\begin{lemma}\label{lem:spec}
 There exists a periodic point $p$ in $U_{l+n}$, of period $l$,  such that $f^{l-r}(p)\in U_{r+n}$, $r=0,1,\dots, l$, and
\begin{equation}\label{eq:spec}
|S_l\eta(p)-S_l\eta(z)|<C \rho(z, f^l(z))^\alpha.
\end{equation}
\end{lemma}

From Lemma~\ref{lem:spec} the estimate   \eqref{eq:holder} follows easily. We have:
$$
\aligned
&u(f^m(x))-u(f^k(x))= \eta(f^{k}(x))+\dots + \eta(f^{m-1}(x))\\
&=S_l\eta(z)-S_l\eta(p)
\endaligned
$$
since $S_l\eta(p)=0$ by Corollary \ref{C:periodic-zero}.
By Lemma~\ref{lem:spec},
$$|S_l\eta(p) - S_l\eta(z)|<C \rho(z, f^l(z))^\alpha,$$
so finally
$$|u(f^m(x))-u(f^k(x))|<C \rho(f^k(x),f^m(x))^\alpha$$
and \eqref{eq:holder} holds, completing the proof of Proposition~\ref{prop:homology}.
\end{proof}

\begin{proof}[Proof of Lemma~\ref{lem:spec}]

We use the fact that the map $f$ uniformly expands balls with respect to the visual metric.
Indeed, let us fix $\epsilon > 0$, and let $\rho$ be its corresponding visual metric. Then, by \cite[Theorem 3.2.1]{HP} and \cite[Proposition 3.2.2]{HP},
there exists $r_0 > 0$ such that for any $r < r_0$ and for any $\xi \in X$  one has $f(B(\xi, r e^{-\epsilon})) = B(f(\xi), r)$
(in the language of \cite{HP}, $r_0 = |F(\xi)|_\epsilon$, and that is constant for any $\xi \in \partial_\epsilon \Gamma$, which is our repellor $X$).

Denote $y = f^l(z)$ and assume $\rho(y,z)<r$ sufficiently small so that we can apply \cite[Proposition 3.2.2]{HP}.
Denote $y_s=f^{m-s}(x)$. Choose consecutive preimages $z_s$ of $z$ such that
$\rho(y_s,z_s)\leq e^{-\epsilon} \rho(y_{s-1}, z_{s-1})$,
for $s=0,1, \dots, l = m-k$. Thus, $\rho(z, z_l) \leq e^{-l \epsilon} \rho(y, z)$.

(Notice that we do not choose just any $z_s$ being a  preimage of $z_{s-1}$
in the pullback $U_s$ of $B_{s-1}:=B(y_{s-1}, r e^{-(s-1)\epsilon})$
containing $y_s$ because $U_s$ can be strictly bigger than $B_s$.
We choose $z_s\in B_s$.)

Now, continue choosing preimages $(z_s)$ for $s \geq l$ close to the previously chosen $z_{s}$.
This way, we obtain $\rho(z_s, z_{s+l}) \leq e^{- s \epsilon} \rho(z, z_l)$ for any $s \geq 0$, hence $p := \lim_{n \to \infty} z_{nl}$ is the required
periodic point.

Indeed we have, for any $h$ with $0 \leq h < l$,
\begin{align*}
\rho(f^h(p), f^h(z)) & = \lim_{n \to \infty} \rho(z_{nl-h}, y_{l-h}) \leq  \\
& \leq \rho(z_{l-h}, y_{l-h}) + \sum_{n = 1}^\infty \rho(z_{nl-h}, z_{(n+1)l-h}) \leq \\
& \leq e^{-\epsilon(l-h)} \rho(y, z) + \sum_{n = 1}^\infty e^{-(nl-h)\epsilon} e^{-l \epsilon} \rho(y, z) = \\
& = \frac{e^{-\epsilon(l-h)}}{1-e^{-1}} \rho(y, z)
\end{align*}
hence, using the H\"older continuity of $\eta$, there exist $K, \beta > 0$ such that
$$|S_l \eta(p) - S_l \eta(z)| \leq \sum_{h = 0}^{l-1} K \rho(f^h (p), f^h(z))^\beta \leq K \sum_{h = 0}^{l-1} \left( \frac{e^{-\epsilon(l-h)}}{1-e^{-1}} \rho(y, z) \right)^\beta \leq C \rho(y, z)^\beta$$
with $C := K (1-e^{-1})^{-\beta} (1 - e^{-\beta \epsilon})^{-1}$.
This completes the proof.
\end{proof}

Proposition~\ref{prop:homology} concludes the proof of Theorem~\ref{T:main} \eqref{itm:cohom} in the introduction. Indeed, given any exponentially contracting metric $\rho$ on $X$ there exists a visual
metric $\rho'$ which induces the same topology, and, by \cite[proof of Theorem 3.2.1, page 58]{HP}, such that the identity map $(X, \rho') \to (X, \rho)$
is Lipschitz.

Then, $\varphi$ and $\psi$ are also H\"older continuous with respect to $\rho'$, and we can apply Proposition~\ref{prop:homology} to $\rho'$. Thus, we obtain
$u$ which is H\"older continuous with respect to the visual metric, which means it is continuous with respect to the original metric $\rho$. Note, however, that we do not know anything about the modulus of continuity of $u$.

A similar argument also yields the proof of Theorem~\ref{T:main} \eqref{itm:sigma0}.
If $\sigma = 0$ and we normalize $\varphi$ so that $\int \varphi \ d\mu = 0$, then by \cite[Lemma 1]{PUZ} we have
$$|S_n (\varphi \circ \pi) (\omega)|<C$$
for any $\omega \in \Sigma$ and any $n \geq 0$, hence by surjectivity also
$$|S_n \varphi (x)|<C$$
for any $x \in X$ and any $n \geq 0$.
Now, as in the proof of Proposition~\ref{prop:homology} (with $\psi = 0$) there exists $u : X \to \mathbb{R}$ which is H\"older continuous with respect to the visual metric on $X$ such that
$$\varphi = u \circ f - u.$$
Thus, $u$ is continuous with respect to the original metric, completing the proof.

Finally, the proofs of Theorem~\ref{T:main-sphere} \eqref{itm:sigma0} and \eqref{itm:cohom} proceed exactly in the same way, by considering a visual metric on $X$ and running the same argument as in the proof of Proposition~\ref{P:periodic} and Proposition~\ref{prop:homology}, just replacing $\pi$ with $\pi \circ \Pi$ where $\Pi : \Sigma \to Y$
is the coding map, $Y$ is the repellor for the blown up space as in Section \ref{S:periodic-crit}, and $\pi : Y \to X$ is as constructed in Proposition~\ref{P:blowup}.

\section{Appendix A. Exponential contraction}

Let us now present a direct construction of an exponentially contracting metric on the blown-up repellor $Y$, 
alternative to the approach via Frink's lemma (Lemma \ref{L:Frink-Y}). 

Let $(X, \rho)$ be the repellor with a metric $\rho$ which is exponentially contracting with respect to a cover $\mathcal{U}$, let $Y$
be the repellor in the blown up space, let $\pi : Y \to X$ be the blowdown map and let $g : Y \to Y$ be the dynamics on $Y$, so that $\pi \circ g = f \circ \pi$.

Let us assume for simplicity that there is only one periodic critical point $p$ in $X$, that $f(p) = p$, and there is no other critical point in the grand orbit of $p$. Let $d$ be the local degree of $f$ at $p$.

According to Lemma \ref{L:conj}, let $U$ be a neighborhood of $p$ in $X$ which is homeomorphic to the unit disc and such that $f : U \cap f^{-1}(U) \to U$ is topologically conjugate to $h(r e^{i\theta}) = \lambda r e^{i d \theta}$ for some $\lambda > 1$.

In the blown up space $Y$, the disc $U$ is replaced by an annulus $A$ which is homeomorphic to $[0, 1) \times \mathbb{R}/\mathbb{Z}$,
and the map $g$ is topologically conjugate to $(r, \theta) \mapsto (\lambda r, d \theta)$. Let $p: A \to \mathbb{R}/\mathbb{Z}$ be the projection to
the angular coordinate.

Since we know a local model for $g$ on $A$, we can define in polar coordinates for each $n \geq 0$ the annulus
$A_n := \{ (r, \theta) \in A \ : \ \lambda^{-n-2} < r < \lambda^{-n} \}$. Then we know that $g(A_{n+1}) = A_n$.  Moreover, let
$B_n := \{ (r, \theta) \in A \ : \ r < \lambda^{-n} \}$.

(Note that all our metrics are only defined on $X$. When we write $\textup{diam}_\rho(A)$, we mean $\textup{diam}_\rho(A \cap X)$.)
By replacing $f$ with a power, we can assume that
\begin{equation}
\label{E:diam}
\textup{diam}_\rho(A_1) < \rho(\partial A_1, Y \setminus A).
\end{equation}
Now, let us define for $0 \leq j \leq d^{n+1} - 1$ the open sets
$$A_{n, j} := \{ (r, \theta) \in A \ : \ \lambda^{-n-2} < r < \lambda^{-n} \textup{ and } j/d^{n+1} < \theta/2\pi < (j+2)/d^{n+1} \}.$$
Note that for any $n > 0$ and any $j$ the map $g^n$ is a homeomorphism from $A_{n, j}$ to some $A_{0, r}$ for $0 \leq r \leq d -1$.
Moreover, let $A^* := Y \setminus \overline{B_1}$, which is an open neighborhood of the complement of $A$ in $Y$.

Let $\rho$ be an exponentially contracting metric on $X$, with respect to an open cover $\mathcal{U}$. Consider the projection map $\pi : Y \to X$, and let $\rho_1(x, y) := \rho(\pi(x), \pi(y))$ the pullback of the metric $\rho$ on $Y$.

On each open set $A_{n,j}$ let us define the pseudo-metric (for $x, y \in Y$)
$$\rho_{A_{n,j}}(x, y) := \lambda^{-n} \rho_1(g^n x, g^n y)$$
with $\lambda = d$. Moreover, on $A^*$ define $\rho_{A^*}(x, y) := 0$. Let $\mathcal{A}$ be the cover which is the union of the $A_{n,j}$ and $A^*$.
Now, given $x,y$, we say an \emph{admissible chain} $\gamma$ between $x, y$ is a sequence $(x_i)_{i = 0}^k$ of points such that
$x = x_0$, $y = x_k$ and for any $0 \leq i \leq k-1$ both points $x_i$ and $x_{i+1}$ lie in the same element $\alpha_i$ of $\mathcal{A}$.
We define the \emph{length} of such an admissible chain as
$$L(\gamma) := \sum_{i = 0}^{k-1} \rho_{\alpha_i}(x_i, x_{i+1}).$$
Finally, we define the pseudo-metric $\rho_A$ on $Y \setminus S_p$ as
$$\rho_A(x,y) := \inf_{\gamma \in \Gamma_{x,y}} L(\gamma)$$
where $\Gamma_{x,y}$ is the set of all admissible chains between $x$ and $y$. Finally, extend $\rho_A$ to $S_p$ by taking its completion.

\medskip
\noindent \textbf{Properties of $\rho_A$.} Since every concatenation of admissible chains is admissible, $\rho_A$ satisfies the triangle inequality.

Moreover, if $x, y \in A_{n,j}$, then by definition
\begin{equation} \label{E:disjoint}
\rho_A(x, y) \leq \rho_{A_{n,j}}(x, y).
\end{equation}

Finally, if $V_1 \subseteq B_2$ is connected and $V_0 = g(V_1) \subseteq B_1$, then
\begin{equation}
\label{E:contr}
\textup{diam}_{\rho_A}(V_1) \leq \lambda^{-1} \textup{diam}_{\rho_A}(V_0).
\end{equation}

\begin{proof}
Let $x, y \in V_0$ and $x', y' \in V_1$ with $g(x') = x, g(y') = y$. Now, fix $\epsilon > 0$ and consider an admissible chain $\gamma$ given by $x = x_0, \dots, x_{k-1}, x_k = y$ between $x$ and $y$,
and such that
$$L(\gamma) \leq \rho_A(x, y) + \epsilon.$$
By \eqref{E:diam}, for $\epsilon$ small such a chain lies entirely in $A$. Hence, by lifting $x_i$ we find a chain
$(y_i)_{i = 0}^k$ with $g(y_i) = x_i$. Then each pair $y_i, y_{i+1}$ lies in a set $A_{n_i + 1, j'_i}$ with $g(A_{n_i + 1, j'_i}) = A_{n_i, j_i}$,
and by definition
$$\rho_{n_i+1, j'_i}(z, w) = \lambda^{-1} \rho_{n_i, j_i}(g(z), g(w)) \qquad \forall z, w \in A_{n_i + 1, j'_i}$$
hence
$$\rho_A(x', y') \leq \sum_{i = 0}^{k-1} \rho_{n_i+1, j'_i}(y_i, y_{i+1}) = \lambda^{-1} \sum_{i = 0}^{k-1}\rho_{n_i, j_i}(x_i, x_{i+1}) \leq
\lambda^{-1} \rho_A(x, y) + \epsilon$$
and the claim follows as $\epsilon \to 0$.
\end{proof}

(Note that \eqref{E:diam} is necessary since the shortest admissible chain between $x$ and $y$ might not lie entirely in $A$,
using a ``shortcut" through the complement of $A$, where $\rho_A$ is zero).

Now, for any $q$ in the grand orbit of $p$, let $k$ be the smallest integer such that $f^k(q) = p$. Let $A(q)$ be the pullback of $A$ by $g^k$.
Define for each $k$ the set $E_k := f^{-k}(p) \setminus f^{-k+1}(p)$, and let $q \in E_k$. If $g^k : A(q) \to A$ is injective,
we define
the pseudo-metric
$$\rho_{A(q)}(x, y) := \rho_A(g^k x, g^k y)$$
which is supported on $A(q)$, and given by pulling back the metric $\rho_A$ on $A = A(p)$.
If $g^k : A(q) \to A$ is not injective, we define
$$\rho_{A(q)}(x, y) := \inf \left( \sum_{i = 1}^r \rho_A(g^k x_i, g^k x_{i+1}) \right)$$
where the infimum is taken over all finite sets $x = x_0, x_1, \dots, x_r = y$ of points such that $x_i$ and $x_{i+1}$ are both contained
in a connected, open set over which $g^k$ is injective.

\medskip
\noindent \textbf{Definition of the metric $\widetilde{\rho}$.} We finally define the metric on $Y$ to be
\begin{equation} \label{E:def-rho}
\widetilde{\rho}(x, y) := \rho_1(x, y) + \sum_{n \geq 0} \sum_{q \in E_n} c^n \rho_{A(q)}(x, y)
\end{equation}
where $c < 1$ is a constant which will be determined later.

Let $\mathcal{U}_n$ be the collection of all connected components of preimages $f^{-n}(U)$ for $U \in \mathcal{U}$.
Since $\rho$ is exponentially contracting, then there exist $C_1, \alpha_1 > 0$ such that
\begin{equation} \label{E:exp-U}
\sup_{U \in \mathcal{U}_n} \textup{diam}_{\rho}(U) \leq C_1 e^{-n \alpha_1}.
\end{equation}
By substituting $\mathcal{U}$ with $\mathcal{U}_{n_0}$ for some $n_0$, we can assume $\textup{diam}_{\rho}(U) < r$ for any $U \in \mathcal{U}_n$
and any $n$.

Choose $k > 0$ such that $\textup{diam}_{\rho_1} A_k < r$.
Moreover,
denote $W_p^{(1)} := N_\epsilon( (r, \theta) \in A \ : \ 0 \leq r \leq \lambda^{-k}, 0 \leq \theta \leq 1/2)$ and $W_p^{(2)} := N_\epsilon( (r, \theta) \in A \ : \
0 \leq r \leq \lambda^{-k}, 1/2 \leq \theta \leq 1)$
where $N_\epsilon$ denotes the $\epsilon$-neighborhood in the metric $\widetilde{\rho}$ and $\epsilon > 0$ is small.

Let us consider the open cover $\mathcal{V} := \{ \pi^{-1}(U) \ : \ U \in \mathcal{U} \} \cup W_p^{(1)} \cup W_p^{(2)}$. We want to
prove that $\widetilde{\rho}$ is exponentially contracting with respect to the cover $\mathcal{V}$.  Let $\mathcal{V}_n$ denote
the set of connected components of the preimages $g^{-n}(V)$ of elements $V$ of $\mathcal{V}$.

\begin{proposition}
There exist constants $C, \alpha > 0$ such that for any $V \in \mathcal{V}_n$ we have
$$\textup{diam}_{\widetilde{\rho}} (V) \leq C e^{-n \alpha}.$$
\end{proposition}

\begin{proof}

The proof is given in several steps.

\textbf{Step 1.} First of all, since $\rho$ is exponentially contracting and $\mathcal{V}$ is a refinement of $\pi^{-1}(\mathcal{U})$, we obtain
\begin{equation} \label{E:exp-rho1-V}
\sup_{V \in \mathcal{V}_n} \textup{diam}_{\rho_1}(V) \leq C_1 e^{-n \alpha_1}.
\end{equation}

\textbf{Step 2.} Now, we prove by induction that there exist $C_2, \alpha_2 > 0$ such that for each $V \in \mathcal{V}_n$,
\begin{equation} \label{E:exp-rhoA-V}
\textup{diam}_{\rho_A} (V) \leq C_2 e^{-n \alpha_2}
\end{equation}
and $p(V \cap A_0) \neq \mathbb{R}/\mathbb{Z}$.
The case $n = 0$ is trivial, as one can just choose the appropriate $C_2$ since the cover $\mathcal{V}$ is finite.

Now, let $V \in \mathcal{V}_n$ and $V' = g(V) \in \mathcal{V}_{n-1}$. There are two cases.
\begin{enumerate}
\item

Let us suppose that $V \subseteq B_2$, so that $V' \subseteq B_1$.
Then by \eqref{E:contr}  we obtain
$$\textup{diam}_{\rho_{A}} (V) \leq \lambda^{-1} \textup{diam}_{\rho_{A}} (V') \leq \lambda^{-1} C_2 e^{-(n-1) \alpha_2} \leq C_2 e^{-n \alpha_2}$$
by taking $\alpha_2 < \log \lambda$.
\item
In the other case, $V$ must be disjoint from $B_{k_0}$ with $k_0 = 3$, hence by \eqref{E:disjoint} and \eqref{E:exp-rho1-V},
for any $k \leq k_0$ we have
$$\textup{diam}_{\rho_{A}} (V \cap A_{k, j}) \leq \textup{diam}_{\rho_{k,j}} (V \cap A_{k,j}) = $$
$$ =  \lambda^{-k} \textup{diam}_{\rho_1}(g^k(V) \cap A_{0, r})  \leq \lambda^{-k} \textup{diam}_{\rho_1}(g^k(V)) \leq \lambda^{-k} C_1 e^{-(n-k) \alpha_1}$$
hence, by considering all $A_{k, j}$ for $k \leq k_0$,
$$\textup{diam}_{\rho_{A}} (V) \leq d^{k_0} \lambda^{-k} C_1 e^{-(n-k) \alpha_1}$$
which is less than $C_2 e^{-n \alpha_2}$ if one takes  $\alpha_2 < \alpha_1$.
\end{enumerate}

\medskip

\textbf{Step 3.} Now, let us consider all other contributions to the sum in \eqref{E:def-rho}, i.e. the terms relative to $\rho_{A(q)}$ for $q \neq p$.
Consider $V \in \mathcal{V}_n$, and let $V_0, V_1, \dots, V_{n-1}, V_n = V$ be a chain of open sets with $V_k \in \mathcal{V}_k$ and $V_{k}$
the connected component of $g^{-1}(V_{k-1})$.

Consider $q \in E_k$. Then, since the degree of $g^k$ is at most $d^k$, where $d$ is the global degree of $f$, by definition of $\rho_{A(q)}$ we have
$$\textup{diam}_{\rho_{A(q)}}(V) \leq d^k \textup{diam}_{\rho_A}(V_{n-k}) \leq C_2 d^k e^{-(n-k) \alpha_2}$$
Note moreover that the set $f^{-k}(p)$ contains at most $d^k$ elements.
Thus, we have for $k \leq n$,
\begin{equation} \label{E:first-part}
\sum_{k = 0}^n  \sum_{q \in E_n} c^k \textup{diam}_{\rho_{A(q)}} (V)  \leq  \sum_{k = 0}^n c^{k} d^{2 k} C_2  e^{-(n-k) \alpha_2} \leq
C_2 e^{-n \alpha_2} \sum_{k = 0}^n (c d^2 e^{\alpha_2})^k
\end{equation}
which is exponentially bounded if $cd^2 e^{\alpha_2} < 1$.

Finally, let $M := \textup{diam}_{\rho_A}(Y)$. Then
\begin{equation} \label{E:second-part}
\sum_{k > n } \sum_{q \in E_n} c^k \textup{diam}_{\rho_{A(q)}} (V)  \leq \sum_{k > n} c^k d^{2k} M \leq C_4 (cd^2)^n
\end{equation}
which is also exponentially bounded if $cd^2 < 1$.

The claim now follows by combining \eqref{E:exp-rho1-V}, \eqref{E:first-part} and \eqref{E:second-part}, by choosing an appropriate value of $c < 1$.
\end{proof}

\section{Appendix B. Embedding into the sphere}

\begin{proposition}
The weakly coarse expanding system $g:\widetilde W_1 \to \widetilde W_0$ in the assertion of Proposition \ref{P:blowup} can be continuously embedded into the sphere $S^2$, where $Y$ becomes a repellor for the extended system.
More precisely, there exist a continuous embedding $\iota : \widetilde{W}_0 \to S^2$, a continuous map $g' : S^2 \to S^2$ with $g' \circ \iota = \iota \circ g$ and an open set $W_0'$ which contains $\iota(\widetilde{W}_0)$ so that $\iota(Y) = \bigcap_{n \geq 0} (g')^{-n}(W_0')$.
\end{proposition}

\begin{proof}
For simplicity, let us assume that there is only one periodic critical point, and that it is actually fixed.
Let $p$ be the fixed point for $f$, and note that $D(p) :=S^2\setminus\{p\}$ is homeomorphic to the unit disc $\D$,
and, if $U = h(\D)$ is the range of the chart constructed in Lemma \ref{L:conj}, then $U \setminus \{p\}$ is homeomorphic to an annulus.
Thus, let us define a homeomorphism $H: \D \to D(p)$, which we can choose so that
it maps the annulus $\{ z \in \D \ : \ |z|>1/2\}$ onto $U \setminus \{ p \}$.

Now, we replace $p$ by a circle $S_p$ and add an open disc $D'(p)$, which we parameterize by $H' : \widehat{\mathbb{C}} \setminus \overline{\D} \to D'(p)$
and we identify to $D(p)$ by some homeomorphism $\phi: D(p) \to D'(p)$ so that $\phi \circ H(z) = H'(1/ \bar z)$ for any $z \in \D$.

Then, we construct the space
$$
\hat S_p:=D(p) \sqcup S_p \sqcup D'(p)
$$
and define a topology on it by the standard neighborhoods
of points in $D(p)$ and $D'(p)$ and by
$$
V_{\epsilon, \alpha,\beta}:=\{(r,\theta): 1-\epsilon<r<1+\epsilon, \alpha<\theta<\beta\}
$$
for any point $(1,\theta)\in S_p$ (note that here $S_p$ is characterized by $r = 1$, differently from Lemma \ref{L:conj}).
Then the polar coordinates from $H, H'$ glue along $|z|=1$ to yield a homeomorphism $\hat{H}_p: S^2\to \hat S_p$.

\medskip

We conclude that  $D(p) \sqcup S_p$ with the topology defined in Section \ref{S:periodic-crit} embeds by $\hat{H}^{-1}_p$ into $S^2$ with an open hole (open disc) removed.

Similarly, we consecutively blow up the points $q$ in the grand orbit $O(p)$ to open discs, each time embedding the result in $S^2$.  
If $m= m(q)$ is the least integer such that $f^m(q)=p$ and $f^m$ has local degree $1$ at $q$, we use the chart $h_q = f^{-m}\circ h$, 
extended symmetrically to cover $D'(q)$; if the local degree is higher, we apply adequate roots. (See also the 
compatible system of charts from Section \ref{S:periodic-crit}, Remark \ref{R:2}).


\medskip

Consider $S^2\setminus O^-(p)$, 
where $O^-(p)$ is the grand orbit of $p$. We arrange all $q\in O^-(p)$ different from $p$
into a sequence $(q_j)$ and perform a sequence of embeddings of $S^2\setminus O^-(p)$ into $S^2$, blowing up $q$'s as above
(the order need not be compatible with $m(q)$; we may forget about the dynamics).

We care that the complementary  $D'(q_j)$ have diameters
in the spherical metric in $S^2$ quickly shrinking to $0$ and  that
consecutive embeddings differ from the preceding ones
on neighbourhoods of $q$ of  diameters also quickly shrinking to $0$, so that they form a Cauchy sequence.
So the limit embedding $\iota$ exists and extends to the closure
$$
\widetilde{S} =\Bigl(S^2\setminus O^-(p)\Bigr) \sqcup \bigsqcup_{q\in O^-(p)} S_q
$$
with image
$S^2\setminus \bigcup_{q\in O(p)}D'(q)$, which is a Sierpi\'nski carpet.

Crucially, it is possible to perform consecutive perturbations of the embedding not changing the embedding close to the previous $q_i$'s, i.e. not moving already constructed $D'(q)$'s.

\medskip

Notice finally that we can extend $g$ from a neighborhood $\widetilde{W}_0$ of $Y$ in $\widetilde S$ to neighbourhoods of $\iota(S_q)$
in the image $S^2$, using in $D'(p)$ the formula symmetric to the one in Lemma \ref{L:conj}.  We proceed similarly at
$q\in f^{-m}(p)$ using the charts $h_q$. So $Y$ becomes a repeller in $S^2$.

In fact the extension can be easily defined on all the $D'(q)$'s, so that $D'(p)$ is the immediate basin of 
attraction to an attracting fixed point and $\bigcup_{q\in O^-(p)} D'(q)$ the full basin of attraction.


\end{proof}


\begin{thebibliography}{99}

\bibitem{ASU}
J. Atnip, H. Sumi, M. Urba\'nski,
\emph{The Dynamics and Geometry of Semi--Hyperbolic
Rational Semigroups}, in preparation.

\bibitem{BCM1}
A. Blokh, C. Cleveland, M. Misiurewicz,
\emph{Expanding polymodials}, in \emph{Modern dynamical systems and applications}, 253--270, 
Cambridge Univ. Press, Cambridge, 2004.

\bibitem{BCM2}
A. Blokh, C. Cleveland, M. Misiurewicz, \emph{Julia sets of expanding polymodials},
Ergodic Theory Dynam. Systems 25 (2005), 1691--1718.

\bibitem{BM}
M. Bonk, D. Meyer, \emph{Expanding Thurston maps}, Mathematical Surveys and Monographs 225,
American Mathematical Society, 2017.

\bibitem{bowen}
R. Bowen,
\emph{Equilibrium States and the Ergodic Theory of Anosov Diffeomorphisms},
Lecture Notes in Mathematics 470, Springer, 1975.

\bibitem{CFP}
J. W. Cannon, W. J. Floyd, W. R. Parry, \emph{Finite subdivision rules},
Conform. Geom. Dyn. 5 (2001), 153--196.

\bibitem{CE}
P. Collet, J.-P. Eckmann,
\emph{Positive Lyapunov exponents and absolute continuity for maps of the interval},
Ergodic Theory Dynam. Systems 3 (1983), 13--46.

\bibitem{ComRiv}
H. Comman, J. Rivera-Letelier, \emph{Large deviation principles for non-uniformly hyperbolic rational maps},
Ergodic Theory Dynam. Systems 31 (2011), no. 2, 321--349.

\bibitem{CoKo}
A. Constantin, B. Kolev, \emph{The theorem of {K}er\'{e}kj\'{a}rt\'{o} on periodic homeomorphisms of the disc and the sphere},
Enseign. Math. (2) 40 (1994), no. 3-4, 193--204.

\bibitem{DK}
M. Denker, M. Kesseb\"ohmer,
\emph{Thermodynamic formalism, large deviation, and multifractals}, in \emph{Stochastic Climate Models},
Progr. Probab. 49 (2001), 159--170.

\bibitem{Engelking}
R. Engelking, \emph{General Topology}, Heldermann, Berlin, 1989. 

\bibitem{ERS}
K. Ero\u{g}lu, S. Rohde, B. Solomyak,
\emph{Quasisymmetric conjugacy between quadratic dynamics and iterated function systems},
Ergodic Theory Dynam. Systems, 30 (2010), 1665--1684.

\bibitem{frink}
A. H. Frink, \emph{Distance functions and metrization problem},
Bull. Amer. Math. Soc. 43 (1937), 133--142.

\bibitem{GHMZ}
Y. Gao, P. Ha\"issinsky, D. Meyer, J. Zeng,
\emph{Invariant Jordan curves of Sierpi\'nski carpet rational maps},
Ergodic Theory Dynam. Systems 38 (2018), no. 2, 583--600.

\bibitem{HP}
P. Ha\"issinsky, K. M. Pilgrim, \emph{Coarse expanding conformal dynamics},
Ast\'erisque 325 (2009), viii+139 pp.

\bibitem{HP-ex}
P. Ha\"issinsky, K. M. Pilgrim, \emph{Examples of coarse expanding conformal maps},
Discrete Contin. Dyn. Syst. 32 (2012), no. 7, 2403--2416.

\bibitem{InoRiv}
I. Inoquio-Renteria, J. Rivera-Letelier, \emph{A characterization of hyperbolic potentials of rational maps},
Bull. Braz. Math. Soc. (N.S.) 43 (2012), no. 1, 99--127.

\bibitem{Li1}
Z. Li, \emph{Weak expansion properties and large deviation principles for expanding Thurston Maps},
Adv. Math. 285 (2015), 515--567.

\bibitem{Li2}
Z. Li, \emph{Periodic points and the measure of maximal entropy of an expanding Thurston map},
Trans. Amer. Math. Soc. 368 (2016), 8955--8999.

\bibitem{Li3}
Z. Li, \emph{Equilibrium states for expanding Thurston maps},
Comm. Math. Phys. 357 (2018), 811--872.

\bibitem{li-book}
Z. Li, \emph{Ergodic theory of expanding Thurston maps},
The Atlantis Series in Dynamical Systems, volume 4, Atlantis Press (Springer), 2017.

\bibitem{Feliks}
F. Przytycki, \emph{Hausdorff dimension of harmonic measure on the boundary of an attractive basin for a holomorphic map},
Invent. Math. 80 (1985), 161--179.

\bibitem{Przytycki-ICM2018}
F. Przytycki, \emph{Thermodynamic formalism methods in one-dimensional real and complex dynamics},
Proceedings of the International Congress of Mathematicians 2018, Rio de Janeiro, vol. 2, 2081--2106.

\bibitem{PU}
F. Przytycki, M. Urba\'nski, \emph{Conformal Fractals: Ergodic Theory Methods},
London Mathematical Society Lecture Note Series (371), Cambridge University Press, 2010.

\bibitem{PUZ}
F. Przytycki, M. Urba\'nski, A. Zdunik.
\emph{Harmonic, Gibbs and Hausdorff Measures on Repellers for Holomorphic Maps, I},
Ann. of Math. (2) 130, no. 1 (1989), 1--40.

\bibitem{Ruelle}
D. Ruelle, \emph{Thermodynamic formalism},
Vol. 5. Encyclopedia of Mathematics and its Applications,
Addison-Wesley, Reading, Mass., 1978.

\bibitem{Simon}
B. Simon, \emph{Real Analysis: A comprehensive course in Analysis, Part I},
American Math. Society, Providence, RI, 2015.

\bibitem{Sinai}
Ja. G. Sinai, \emph{Gibbs measures in ergodic theory},
Uspekhi Mat. Nauk. 27.4 (166) (1972), 21--64.

\bibitem{Spanier}
E. H. Spanier, \emph{Algebraic Topology},
McGraw Hill Inc., 1966.

\bibitem{Tsujii}
M. Tsujii, \emph{Positive Lyapunov exponents in families of one dimensional dynamical systems},
Invent. Math. 111 (1993), 113--137.

\bibitem{Wa}
P. Walters, \emph{Relative Pressure, Relative Equilibrium States, Compensation Functions and Many-to-One Codes Between Subshifts},
Trans. Amer. Math. Soc. 296 (1986), no. 1,  1--31.

\bibitem{Why}
G. T. Whyburn, \emph{Open Mappings on 2-Dimensional Manifolds},
Journal of Mathematics and Mechanics 10 (1961), no. 1, 181--197.

\bibitem{Why2}
G. T. Whyburn, \emph{Analytic topology}, 
American Mathematical Society, New York, 1942.

\end{thebibliography}
\end{document}